\documentclass{amsart}
\usepackage{amsmath}
\usepackage{amssymb}
\usepackage{amscd}
\usepackage{amsfonts}
\usepackage{latexsym}
\usepackage{amsfonts}
\usepackage{graphicx}
\usepackage[all]{xy}
\usepackage{mathrsfs}
\usepackage{amscd}
\usepackage{verbatim}
\usepackage{url}
\usepackage{stmaryrd}

\usepackage[usenames]{color}

\newcommand{\scrE}{{\mathscr{E}}}

\newcommand{\sfF}{{\mathsf{F}}}
\newcommand{\sfV}{{\mathsf{V}}}

\numberwithin{equation}{section}

\newcommand{\scrExt}{{\scrE}xt}

\newcommand{\fg}{{\mathfrak{g}}}

\newcommand{\fl}{{\mathfrak{l}}}

\newcommand{\fo}{{\mathfrak{o}}}
\newcommand{\fs}{{\mathfrak{s}}}

\newcommand{\sX}{{\mathscr X}}
\newcommand{\E}{\mathcal{E}}

\newcommand{\cC}{\mathcal{C}}
\newcommand{\cF}{\mathcal{F}}
\newcommand{\cG}{\mathcal{G}}
\newcommand{\cH}{\mathcal{H}}
\newcommand{\cL}{\mathcal{L}}

\newcommand{\cN}{\mathcal{N}}
\newcommand{\cO}{\mathcal{O}}
\newcommand{\cT}{\mathcal{T}}

\newcommand{\dbQ}{{\mathbb{Q}}}

\newcommand{\dbZ}{{\mathbb{Z}}}

\newcommand{\BA}{{\mathbb{A}}}

\newcommand{\BG}{{\mathbb{G}}}
\newcommand{\BF}{{\mathbb{F}}}

\newcommand{\BN}{{\mathbb{N}}}
\newcommand{\BP}{{\mathbb{P}}}
\newcommand{\BQ}{{\mathbb{Q}}}

\newcommand{\BZ}{{\mathbb{Z}}}

\newcommand{\fp}{{\mathfrak p}}

\newcommand{\tX}{{\tilde X}}

\newcommand{\BQbar}{\overline{\dbQ}}

\newcommand{\BZbar}{\overline{\dbZ}}

\newcommand{\iso}{\buildrel{\sim}\over{\longrightarrow}}
\newcommand{\epi}{\twoheadrightarrow}
\newcommand{\mono}{\hookrightarrow}

\DeclareMathOperator{\ab}{{ab}}
\DeclareMathOperator{\Ad}{{Ad}}
\DeclareMathOperator{\cont}{{cont}}

\DeclareMathOperator{\Br}{{Br}}

\DeclareMathOperator{\Cone}{{Cone}}
\DeclareMathOperator{\Cocone}{{Cocone}}
\DeclareMathOperator{\Crys}{{Crys}}
\DeclareMathOperator{\crys}{{crys}}

\DeclareMathOperator{\Forg}{{Forg}}
\DeclareMathOperator{\Frac}{{Frac}}

\DeclareMathOperator{\Spec}{{Spec}}

\DeclareMathOperator{\Gal}{{Gal}}

\DeclareMathOperator{\et}{{et}}

\DeclareMathOperator{\Hom}{{Hom}}
\DeclareMathOperator{\Hyp}{{Hyp}}
\DeclareMathOperator{\Isoc}{{Isoc}}

\DeclareMathOperator{\FIsoc}{\mathit{F}-Isoc}
\DeclareMathOperator{\FIsocd}{\mathit{F}-Isoc^{\dagger}}

\DeclareMathOperator{\Lat}{{Lat}}
\DeclareMathOperator{\Lie}{{Lie}}

\DeclareMathOperator{\Newt}{{Newt}}
\DeclareMathOperator{\Norm}{{Norm}}

\DeclareMathOperator{\Rep}{{Rep}}

\DeclareMathOperator{\s}{{ss}}
\DeclareMathOperator{\Sat}{{Sat}}
\DeclareMathOperator{\smooth}{{smooth}}
\DeclareMathOperator{\Stab}{{Stab}}

\newcommand{\HOM}{\underline{\operatorname{Hom}}}
\DeclareMathOperator{\Mor}{{Mor}}
\DeclareMathOperator{\End}{{End}}
\DeclareMathOperator{\Ext}{{Ext}}
\DeclareMathOperator{\Aut}{{Aut}}

\DeclareMathOperator{\Ker}{{Ker}}
\DeclareMathOperator{\Fr}{{Fr}}
\DeclareMathOperator{\im}{{Im}}

\DeclareMathOperator{\perf}{{perf}}

\DeclareMathOperator{\Vect}{{Vec}}

\newtheorem{cor}[subsubsection]{Corollary}
\newtheorem{lem}[subsubsection]{Lemma}
\newtheorem{prop}[subsubsection]{Proposition}

\newtheorem{theorem}[subsubsection]{Theorem}

\newtheorem{quest}[subsubsection]{Question}

\theoremstyle{remark}
\newtheorem{rem}[subsubsection]{Remark}

\theoremstyle{definition}
\newtheorem{ex}[subsubsection]{Example}

\newcommand{\SupSet}{\raise1.75pt
     \hbox{$\,\,\scriptstyle\supset\,$}}
\newcommand{\SubSet}{\raise1.75pt
     \hbox{$\,\,\scriptstyle\subset\,$}}

\catcode`\@=11
\renewcommand{\over}{\@@over}
\renewcommand{\atop}{\@@atop}
\renewcommand{\above}{\@@above}
\renewcommand{\overwithdelims}{\@@overwithdelims}
\renewcommand{\atopwithdelims}{\@@atopwithdelims}
\renewcommand{\abovewithdelims}{\@@abovewithdelims}
\catcode`\@=13

\begin{document} 
\title[Slopes of indecomposable $F$-isocrystals]{Slopes of indecomposable $F$-isocrystals}
\author{Vladimir Drinfeld and Kiran S.~Kedlaya}
\dedicatory{To Yuri Ivanovich Manin with gratitude}

\begin{abstract}
We prove that for an indecomposable convergent or overconvergent $F$-isocrystal on a smooth irreducible variety over a perfect field of characteristic $p$, the gap between consecutive slopes at the generic point cannot exceed 1. (This may be thought of as a crystalline analogue of the following consequence of Griffiths transversality: for an indecomposable variation of complex Hodge structures, there cannot be a gap between non-zero Hodge numbers.) As an application, we deduce a refinement of a result of V.~Lafforgue on the slopes of Frobenius of an $\ell$-adic local system. 

We also prove similar statements for $G$-local systems (crystalline and $\ell$-adic ones), where $G$ is a reductive group.

We translate our results on local systems into properties of the $p$-adic absolute values of the Hecke eigenvalues of a cuspidal automorphic representation of a reductive group over the adeles of a global field of characteristic $p>0$.
\end{abstract}


\maketitle

\tableofcontents

\section{Introduction and main results}
Let $k$ be a perfect field of characteristic $p>0$. Let $X$ be a smooth irreducible quasi-compact scheme over $k$. Let $|X|$ denote the set of closed points of $X$. Let $\eta\in X$ be the generic point.

\subsection{The main theorem}
\subsubsection{Convergent and overconvergent $F$-isocrystals}   \label{sss:conv vs overconv}
To varieties over $k$, one associates a family of Weil cohomology theories indexed by primes $\ell$, consisting of $\ell$-adic \'etale cohomology for $\ell \neq p$  and Berthelot's \emph{rigid cohomology} \cite{LeS} for $\ell = p$.
For $X$ as above, the category of $\BQ_{\ell}$-local systems has not one but two $p$-adic analogues;
these are the categories of \emph{convergent $F$-isocrystals} and \emph{overconvergent $F$-isocrystals}\footnote{An overview of the theory of both types of $F$-isocrystals is given in \cite{Ke6}. The precise definitions of  $\FIsoc(X)$ and $\FIsocd (X)$ can be found in \cite[\S 2]{O} and  \cite[\S 2.3]{Ber}, respectively. Let us note that in the word ``$F$-isocrystal" the letter $F$ stands for the Frobenius corresponding to $\BF_p$ (i.e., raising to the power of $p$). Let us add that $\FIsoc( k)=\FIsoc(\Spec k)=\FIsocd (\Spec k)$ is just the category of finite-dimensional vector spaces $V$ over $\Frac W(k)$ equipped with a $\sigma$-linear isomorphism $F:V\iso V$, where $\sigma\in\Aut W(k)$ is the unique automorphism such that $\sigma (x)\equiv x^p \mbox { mod } p$.} on $X$, respectively denoted by $\FIsoc(X)$ and $\FIsocd (X)$. 
Roughly speaking, convergent $F$-isocrystals are defined (locally) using the Raynaud generic fiber of a $p$-adic lift of $X$ (e.g., for $X = \BA^1_k$, take the closed unit disc over $\Frac W(k)$), whereas overconvergent $F$-isocrystals are defined on some slightly larger region (e.g., a disc of radius greater than 1).

In particular, there is a canonical restriction functor $\FIsocd(X)\to \FIsoc (X)$. It  is known to be \emph{fully faithful} (the proof of this fact is not straightforward, see  \cite[Theorem~1.1]{Ke2} or Theorem~\ref{T:fully faithful1} herein); we thus view $\FIsocd(X)$ as a full subcategory of $\FIsoc (X)$. 

\begin{rem}
An object of $\FIsocd (X)$ is indecomposable in $\FIsocd (X)$ if and only if it is indecomposable in $\FIsoc (X)$. This is a particular case of the following lemma.
\end{rem}

\begin{lem}   \label{l:indecomposability}
Let $F:\cC\to\cC'$ be a fully faithful functor between abelian categories. Then an object $M\in\cC$ is indecomposable if and only if $F(M)$ is.
\end{lem}

\begin{proof}
Indecomposability of $M$ (resp.~$F(M)$) means that $\End M$ (resp.~$\End F(M)$) has no non-trivial idempotents. On the other hand, $\End F(M)\simeq\End M$ by full faithfulness.
\end{proof}

\subsubsection{Slopes}   \label{sss:Slopes}
Let $M\in \FIsoc(X)$ have rank $n$. Then for any $x\in X$ one has numbers $a_i^x(M)\in\BQ$, $1\le i\le n$, called the \emph{slopes} of $M$ at $x$ (see \cite[\S 1.3]{K}). We order them so that $a_i^x(M)\ge a_{i+1}^x(M)$. One can think of the collection of slopes at a fixed $x\in X$ as a dominant rational coweight of the group $GL(n)$. 

Let us recall the definition of slopes. For any $x\in X$, let $x_{\perf}$ denote the spectrum of the perfection of the residue field of $x$. Then $\FIsoc (x_{\perf})$ has a canonical $\BQ$-grading. In particular, the pullback of 
$M\in\FIsoc (X)$ to $x_{\perf}$ is $\BQ$-graded. Let $d_r$ denote the rank of its component of degree $r\in\BQ$. The numbers $a_i^x(M)$ are characterized by the following property: each $r\in\BQ$ occurs among them $d_r$ times.

\medskip

Here is our main result. 

\begin{theorem}  \label{t:main}
Let $\eta\in X$ be the generic point. Let $M\in\FIsoc(X)$ be indecomposable and of rank~$n$. Then $a_i^\eta (M)-a_{i+1}^\eta (M)\le 1$ for all $i\in \{ 1,\ldots , n-1\}$.
\end{theorem}

The proof will be given in \S\ref{s:proof of main}.  In fact, it will be shown that Theorem~\ref{t:main} easily follows from full faithfulness of the restriction functor $\FIsoc(X)\to\FIsoc(U)$, where $U\subset X$ is a dense open subset. The latter statement (Theorem~\ref{t:fully faithful}) immediately follows from previously known results; however, the proofs of these results are difficult. (Hopefully, J.~Kramer-Miller's theory of $F$-isocrystals with logarithmic decay~\cite{KM} will provide an easier proof of Theorem 1.1.5, which bypasses some of these difficult results.)

\begin{rem}
Theorem~\ref{t:main} may be viewed as an analogue for $F$-isocrystals of the following consequence of Griffiths transversality: for an indecomposable variation of complex Hodge structures, there cannot be a gap between non-zero Hodge numbers. 
The local version of this observation is an unpublished result from the second author's PhD thesis \cite[\S 5]{Ke}.
\end{rem}

The number $\sum\limits_{i=1}^n a_i^x(M)$ is the slope of $\det M$ at $x$; it is well known that this number does not depend on $x\in X$.
Set $A(M):=\frac{1}{n}\cdot\sum\limits_{i=1}^n a_i^x(M)$.

\begin{cor}   \label{c:any x}
In the situation of Theorem~\ref{t:main}, for all $x\in X$ one has
\begin{equation}   \label{e:any x}
\sum\limits_{i=1}^r a_i^x(M)-rA(M)\le r(n-r)/2 \quad \mbox{ for all } r\in \{ 1,\ldots , n-1\}.
\end{equation}
\end{cor}

\begin{rem} \label{r:r(n-r)}
The meaning of $r(n-r)/2$ is as follows: $r(n-r)/2=\sum\limits_{i=1}^r c_i$, where $c_1,\ldots, c_n$ are the numbers such that
$\sum\limits_{i=1}^n c_i=0$ and $c_i-c_{i+1}=1$ for $i\in \{ 1,\ldots , n-1\}$. In fact, $c_i=\frac{n+1}{2}-i$.
\end{rem}

\begin{proof}[Proof of Corollary~\ref{c:any x}]
The function $x\mapsto \sum\limits_{i=1}^r a_i^x(M)$ is known to be lower semicontinuous (by semicontinuity of the Newton polygon, see \cite[Cor.~2.3.2]{K}). So it suffices to check \eqref{e:any x} for $x=\eta$. 

Set $b_i:=a_i^\eta (M)-A(M)-c_i$, where $c_1\, ,\ldots, c_n$ are as in Remark~\ref{r:r(n-r)}. We have
to check that $\sum\limits_{i=1}^r b_i\le 0$. It is clear that $\sum\limits_{i=1}^n b_i=0$, and Theorem \ref{t:main} tells us that $b_i\le b_{i+1}$\,. So $n\sum\limits_{i=1}^r b_i=n\sum\limits_{i=1}^r b_i-r\sum\limits_{i=1}^n b_i=\sum\limits_{i=1}^r\sum\limits_{j=r+1}^n (b_i-b_j) \le 0$.
\end{proof}

\begin{rem}
Let $\check\omega^x (M)$ denote the dominant rational coweight of $SL(n)$ corresponding to the numbers $a_i^x(M)-A(M)$, $1\le i\le n$.
Let $\check\rho$ denote the sum of the fundamental coweights of $SL(n)$. Corollary~\ref{c:any x} says that $\check\rho-\check\omega^x (M)$ belongs to the ``positive cone" (i.e., the convex cone generated by the simple coroots). Theorem~\ref{t:main} says that $\check\rho-\check\omega^\eta (M)$ belongs to the dominant cone (which is \emph{strictly} contained in the positive cone if $n\ge 3$).
\end{rem}

\subsection{Counterexamples}   \label{ss:counterex}
One can ask whether in the situation of Theorem~\ref{t:main} the inequality 
$a_i^x (M)-a_{i+1}^x (M)\le 1$ holds for \emph{all} $x\in X$ and $i\in \{ 1,\ldots , n-1\}$. If $n=2$ the answer is ``yes" by Corollary~\ref{c:any x} because in this case $a_1^x (M)-a_2^x (M)=2(a_1^x (M)-A(M))$. In general, the answer is \emph{no}. In Appendix~\ref{s:Dwork} we construct 
counterexamples for $n\in\{3,4\}$ using hypergeometric local systems in the sense of N.~Katz. 

\subsection{An application to $\ell$-adic local systems} \label{ss:l-adic}
Now assume that the ground field $k$ is \emph{finite}. Let $|X|$ denote the set of closed points of $X$. For $x\in |X|$ let $\deg x$ denote the degree over $\BF_p$ of the residue field of $x$. 

Fix an algebraic closure $\BQbar_{\ell}\supset\BQ_{\ell}$, and let $\BQbar$ denote the algebraic closure of $\BQ$ in $\BQbar_{\ell}$. 

\subsubsection{Algebraicity for $\BQbar_{\ell}$-sheaves}  \label{sss:algebraicity}
If $\E$ is a $\BQbar_{\ell}$-sheaf on $X$, $x\in |X|$, and $\bar x$ is a geometric point of $X$ with image $x$, then one can consider the eigenvalues of the geometric Frobenius acting on the stalk $\E_{\bar x}$; for brevity, we will call them 
``Frobenius eigenvalues of $\E_x$" (or ``Frobenius eigenvalues of $\E$ at $x$").

We say that $\E$ is \emph{algebraic} if the Frobenius eigenvalues of $\E_x$ are in $\BQbar$ for every $x$. It is known that any indecomposable   $\BQbar_{\ell}$-sheaf becomes algebraic after tensoring by the pullback of some local system on $\Spec\BF_p$ (for lisse sheaves this is \cite[Cor.~ VII.8]{La}; in general see, e.g., \cite[Cor.~B.8]{Dr}). So algebraicity is a mild assumption.

\subsubsection{Slopes for algebraic lisse $\BQbar_{\ell}$-sheaves}  \label{sss:l-adic slopes}
Fix a valuation $v:\BQbar^\times\to\BQ$ such that $v(p)=1$. Let $\E$ be an algebraic lisse $\BQbar_{\ell}$-sheaf on $X$ of rank $n$.

By assumption, for each $x\in |X|$ the Frobenius eigenvalues of $\E_x$ are in $\BQbar^\times$.
Applying to them the map $v:\BQbar^\times\to\BQ$ and dividing by $\deg x$, one gets $n$ rational numbers. We call them the \emph{slopes} of $\E$ at $x$. We denote them by $a_i^x(\E )$; as before, we order them so that $a_i^x(\E )\ge a_{i+1}^x(\E )$. 

The number $A(\E):=\frac{1}{n}\cdot\sum\limits_{i=1}^n a_i^x(\E)$ does not depend on $x\in |X|$: indeed, by \cite[Prop.~1.3.4(i)]{De} there exists $m\in\BN$ such that $(\det\E)^{\otimes m}$ is a pullback of a rank 1 local system on $\Spec\BF_p$.

\begin{theorem}  \label{t:l-adic}
Let $\E$ be a  lisse $\BQbar_{\ell}$-sheaf on $X$ of rank $n$, which is algebraic in the sense of \S\ref{sss:algebraicity}. 

(i) There exists a unique $n$-uple of rational numbers $a_1^\eta (\E )\ge a_2^\eta (\E )\ge\ldots \ge a_n^\eta (\E )$ with the following property: let $U$ denote the set of all $x\in |X|$ such that $a_i^x (\E )= a_i^\eta (\E )$ for all $i$, then $U$ is non-empty, and for any curve $C\subset X$ the subset $U\cap |C|$ is open in $|C|$.

(ii)  For all  $x\in |X|$ and  $r\in \{ 1,\ldots ,n\}$ one has
\[
\sum\limits_{i=1}^r a_i^x(\E )\le \sum\limits_{i=1}^r a_i^\eta (\E ).
\]

(iii) If $\E$ is indecomposable then $a_i^\eta (\E )-a_{i+1}^\eta (\E )\le 1$ for all $i\in \{ 1,\ldots , n-1\}$.
\end{theorem}

We will prove Theorem~\ref{t:l-adic} in \S\ref{s:proof of l-adic} by combining Theorem~\ref{t:main} with the existence of crystalline companions (a.k.a. ``petits camarades cristallins") proved by T.~Abe \cite{A}.

\begin{rem}
In statement (i) uniqueness is easy (it follows from Lemma~\ref{l:2irreducibility} below).
\end{rem}

\begin{rem}
See \cite{Ke7} for some stronger assertions about the set $U$. In particular, $U$ is open.
\end{rem}

\begin{rem}   \label{r:weak estimate}
Similarly to the proof of Corollary~\ref{c:any x}, statements (ii) and (iii) imply that if $\E$ is indecomposable then
\begin{equation}   \label{e:2any x}
\sum\limits_{i=1}^r a_i^x(\E )-rA(\E )\le r(n-r)/2 
\end{equation}
for all $x\in |X|$ and $r\in \{ 1,\ldots , n-1\}$. At least for irreducible $\E$, this inequality was proved by V.~Lafforgue without using isocrystals, see \cite[Cor.~2.2]{VLa}; we recall his proof in \S\ref{sss:Vincent's proof}. 
A weaker inequality had been conjectured by Deligne, see Conjecture~1.2.10(iv) of~\cite{De}.
\end{rem}

\begin{rem}   \label{r:not guaranteed}
As in \S\ref{ss:counterex}, in the situation of Theorem~\ref{t:l-adic}(iii)
it can happen that 
$a_i^x (\E )-a_{i+1}^x (\E ) > 1$ for some $x\in |X|$ and $i\in \{ 1,\ldots , n-1\}$. Examples (with $\E$ irreducible and $n\in\{ 3,4\}$) are given in Appendix~\ref{s:Dwork}.
\end{rem}

\subsection{Generalization to arbitrary reductive groups}
Theorems~\ref{t:main} and \ref{t:l-adic} are about $GL(n)$-local systems (crystalline and $\ell$-adic ones).
We deduce from them similar statements for $G$-local systems, where $G$ is a reductive group (see Proposition~\ref{p:semicontinuity}, Theorem~\ref{t:2l-adic}, and Theorem~\ref{t:G-version of main}). We allow $G$ to be disconnected; this is convenient for applications to automorphic representations in \S\ref{s:2automorphic} (where $G$ appears as the Langlands dual of a given connected reductive group).

\subsection{An application to automorphic representations}   \label{sss:application automorphic}
V.~Lafforgue \cite{VLa} used automorphic representations and the Langlands correspondence to prove \eqref{e:2any x}. 
Similarly, we use the Langlands correspondence to translate Theorem~\ref{t:main} into properties of the $p$-adic absolute values of the Hecke eigenvalues of a cuspidal automorphic representation of $GL(n,\BA_F)$, where $\BA_F$ is the ring of adeles of a global field $F$ of characteristic $p>0$, see Theorem~\ref{t:automorphic}, \S\ref{sss:Hecke reformulation}, and Example~\ref{e:PGL(3)}. We do not know whether these properties can be proved directly (i.e., without passing to $F$-isocrystals).

A part of Theorem~\ref{t:automorphic} generalizes to automorphic representations of $G(\BA_F )$, where 
$F$ is as above and $G$ is any reductive group over $F$, see Theorem~\ref{t:2automorphic}(i-ii). We are unable to generalize to arbitrary reductive groups the other part of Theorem~\ref{t:automorphic} (namely, the estimate for the generic slope of automorphic representations). However, Theorem~\ref{t:2automorphic}(iii) says that such a generalization would follow from Conjecture~12.7 of \cite{VLa2} (which goes back to J.~Arthur). The proof of Theorem~\ref{t:2automorphic} uses the main theorem of V.~Lafforgue's article \cite{VLa2}.

\subsection{Organization of the article}
In \S\ref{s:fully faithful} we combine some statements from the literature to show that for any open dense $U\subset X$, the restriction functor $\FIsoc (X)\to \FIsoc (U)$ is fully faithful. This result 
plays a crucial role in the proof of our main Theorem~\ref{t:main}.

In \S\ref{s:proof of main} we prove Theorem~\ref{t:main}. 
In \S\ref{s:reformulations} we discuss some equivalent reformulations of Theorem~\ref{t:main}.
In~\S\ref{s:proof of l-adic} we prove Theorem~\ref{t:l-adic}.
In \S\ref{s:automorphic} we discuss the application to automorphic representations of $GL(n)$ mentioned in 
\S\ref{sss:application automorphic}. 

In \S\ref{s:alggroups} we prove some lemmas on algebraic groups. In \S\ref{s:l-adic for G}-\ref{s:FIsoc_G} they are used to prove the generalizations of Theorems~\ref{t:main} and~\ref{t:l-adic} to arbitrary reductive groups.
In \S\ref{s:2automorphic} we treat the slopes of automorphic representations of arbitrary reductive groups by combining the results of \S\ref{s:l-adic for G} with the main theorem of \cite{VLa2}.

In Appendix~\ref{s:Dwork} we provide the counterexamples promised in \S\ref{ss:counterex}.

In Appendix~\ref{s:Crew} we recall R.~Crew's results on $\FIsoc (X)$ and $\FIsocd (X)$ as Tannakian categories.

\subsection{Acknowledgements}
We thank T.~Abe, A.~Beilinson, D.~Caro, H.~Esnault, K.~Kato, N.~Katz, L.~Illusie, A.~Ogus, A.~Petrov,  P.~Scholze, and V.~Vo\-logodsky for valuable advice and references.

Our research was partially supported by NSF grants DMS-1303100 (V.D.) and
DMS-1501214 (K.K.).

\section{Full faithfulness of restriction functors} \label{s:fully faithful}

We reprise part of the discussion in \cite[\S 5]{Ke6} around the full faithfulness of various restriction functors.

\subsection{Partial overconvergence}
 In addition to the two categories of isocrystals considered so far, we will need a third one: for $U$ an open dense subset of $X$,
let $\FIsoc (U,X)$ denote the category of $F$-isocrystals on $U$ overconvergent within $X$.
In particular, we have $\FIsoc(U,X) = \FIsoc(U)$ if $U = X$ and $\FIsoc(U,X) = \FIsocd(U)$
if $X$ is proper over $k$.

\subsection{Full faithfulness}

\begin{theorem} \label{T:fully faithful0}
For any open dense $U\subset X$, the restriction functor $\FIsoc (X)\to \FIsoc (U,X)$ is fully faithful.
\end{theorem}
\begin{proof}
See \cite[Theorem~5.2.1]{Ke4}.
\end{proof}

\begin{theorem} \label{T:fully faithful1}
For any open dense $U\subset X$, the restriction functor $\FIsoc (U,X)\to \FIsoc (U)$ is fully faithful.
(This remains true even if $X$ is not required to be smooth.)
\end{theorem}
\begin{proof}
In the case where $X$ is proper over $k$,
this becomes the statement that $\FIsoc^\dagger(U) \to \FIsoc(U)$ is fully faithful,
which is \cite[Theorem~1.1]{Ke2}.
For the general case, see \cite[Theorem~4.2.1]{Ke5}.
\end{proof}

By combining the preceding results, we obtain the following.
\begin{theorem}   \label{t:fully faithful}
For any open dense $U\subset X$, the restriction functor $\FIsoc (X)\to \FIsoc (U)$ is fully faithful.
\end{theorem}
\begin{proof}
Write the functor as a composition $\FIsoc (X)\to\FIsoc (U,X)\to\FIsoc (U)$. These functors are fully faithful by Theorem~\ref{T:fully faithful0} and Theorem~\ref{T:fully faithful1}, respectively. This completes the proof.
\end{proof}

\section{Proof of Theorem~\ref{t:main}} \label{s:proof of main}
\subsection{Reduction of Theorem~\ref{t:main} to Proposition~\ref{p:coh-vanishing}(a)}

By Theorem~\ref{t:fully faithful} and Lemma~\ref{l:indecomposability}, if $M \in \FIsoc(X)$ is indecomposable, then so is its restriction to any non-empty open subset of $X$. By semicontinuity of the Newton polygon, we may reduce Theorem~\ref{t:main} to the case where the Newton polygon of $M$ is the same at all $x\in X$.  In this case $M$ admits a slope filtration\footnote{E.g., see \cite[Thm.~4.1 and Cor.~4.2]{Ke6} as well as \cite[Thm.~5.1 and Rem.~5.2]{Ke6}. In the case that $\dim X=1$ the slope filtration goes back to N.~Katz \cite[Corollary~2.6.3]{K}.}, so Theorem~\ref{t:main} reduces to the following statement in the spirit of  \cite[Theorem~5.2.1]{Ke}.
 
 \begin{prop}   \label{p:Ext-vanishing}
 Let $M_1,M_2\in\FIsoc (X)$ and $s_1,s_2\in\BQ$. Suppose that $M_i$ is isoclinic of slope $s_i$ at each point of $X$. If $s_1-s_2>1$, then $\Ext^1 (M_1,M_2)=0$.
 \end{prop}
 
A proof of Proposition~\ref{p:Ext-vanishing} is given below. A slightly different proof, more directly based on \cite{Ke}, is given in \cite[Appendix~A]{Ke6}.
 
It is known\footnote{See \cite[\S 2.4]{Ber}. The main point is that $F$-equivariant isocrystals are automatically convergent; this is proved in \cite{Ber} using an argument which goes back to Dwork.} that $\FIsoc(X)$ identifies with the category of Frobenius-equivariant objects in the category $\Isoc (X):=\Crys (X)\otimes\BQ$, where $\Crys  (X)$ is the category of crystals of coherent sheaves on $X$. For $M\in\Isoc (X)$ we set $R\Gamma_{\crys} (X,M):=R\Gamma_{\crys} (X,M_0)\otimes\BQ$, where $M_0$ is any object of $\Crys (X)$ equipped with an isomorphism $M_0\otimes\BQ\iso M$. If $M\in\FIsoc(X)$ then the complex of $\BQ_p$-vector spaces $R\Gamma_{\crys} (X,M)$ is equipped with an action of the Frobenius endo\-mor\-phism~$F$.

\begin{lem}   \label{l:Ext via cocone}
$\Ext^1 (M_1,M_2)$ is canonically isomorphic to the first cohomology of the complex
\begin{equation}  \label{e:Cocone}
\Cocone (R\Gamma_{\crys} (X,N)\overset{F-1}\longrightarrow R\Gamma_{\crys} (X,N)),
\end{equation}
where $N:=\HOM (M_1,M_2)=M_1^*\otimes M_2$ and $\Cocone:=\Cone [-1]$. \qed
\end{lem}

The lemma is well known. We give a proof for completeness. Note that for our purposes it is enough to know that $\Ext^1(M_1,M_2)$ is isomorphic to the first cohomology of (3.1) in the case that $X$ is \emph{affine}, and this weaker statement can be easily checked by choosing a lift of $X$ to a smooth formal scheme over~$W(k)$.

\begin{proof}
Let $\cO$ denote the unit object of the tensor category $\FIsoc (X)$. (Later we will use the same symbol for the unit objects of some other tensor categories.)
The composition $$\Ext^1 (\cO,N)\to\Ext^1 (M_1,M_1\otimes N)\to\Ext^1 (M_1,M_2)$$ is an isomorphism: its inverse is the composition
$$\Ext^1 (M_1,M_2)\to\Ext^1 (M_1\otimes M_1^*,M_2\otimes M_1^*)=\Ext^1 (M_1\otimes M_1^*,N)\to\Ext^1 (\cO,N).$$ So it remains to compute 
$\Ext^1 (\cO,N)$.

Let $\Isoc (X)$ denote the category of isocrystals on $X$. Let $\scrExt_{\FIsoc} (\cO ,N)$ (resp.~$\scrExt_{\Isoc} (\cO ,N)$) denote the Picard groupoid of extensions of $\cO$ by $N$ in the category $\FIsoc (X)$ (resp.~$\Isoc (X)$).

Let us first compute $\scrExt_{\Isoc} (\cO ,N)$.
By definition, $\Isoc (X)=\Crys (X)\otimes\BQ=\Crys_{\BZ_p{\mbox{-flat}}} (X)\otimes\BQ$, where $\Crys (X)$ is the category of crystals of 
coherent sheaves on $X$ and $\Crys_{\BZ_p{\mbox{-flat}}} (X)$ is the full subcategory of those objects of $\Crys (X)$ that have no non-zero subcrystals killed by $p$. The fiber over $N$ of the functor $\Crys_{\BZ_p{\mbox{-flat}}} (X)\to\Isoc (X)$ is the poset of \emph{lattices} in $N$, denoted by $\Lat (N)$. One has
\[
\scrExt_{\Isoc} (\cO ,N)=\underset{L\in \Lat (N)}{\underset{\longrightarrow}\lim} \scrExt_{\Crys} (\cO ,L)
\]
(in the right-hand side $\cO$ denotes the unit object of $\Crys (X)$). 
An object of $\scrExt_{\Crys} (\cO ,L)$ is the same as an $L$-torsor on the crystalline site of $X$, so the Picard groupoid 
$\scrExt_{\Crys} (\cO,L)$ corresponds (in the sense of \cite[Expos\'e XVIII, \S 1.4]{SGA4}) to the complex $\tau^{\le 0}R\Gamma_{\crys} (X,L[1])$. Therefore $\scrExt_{\Isoc} (\cO ,N)$ corresponds to $\tau^{\le 0}R\Gamma_{\crys} (X,N[1])$.

The Picard groupoid $\scrExt_{\FIsoc} (\cO ,N)$ is the groupoid of Frobenius-equivariant objects of $\scrExt_{\Isoc} (\cO ,N)$, so it corresponds to the complex
\[
\tau^{\le 0}\Cocone (R\Gamma_{\crys} (X,N[1])\overset{F-1}\longrightarrow R\Gamma_{\crys} (X,N[1])).
\]
Therefore $\Ext^1 (\cO,N)$ is the $0$-th cohomology of this complex or equivalently, the first cohomology of the complex \eqref{e:Cocone}.
\end{proof}

Lemma~\ref{l:Ext via cocone} shows that Proposition~\ref{p:Ext-vanishing} is a particular case of the following statement in the spirit of  \cite[\S 5.4]{Ke3}.

\begin{prop}  \label{p:0coh-vanishing}
Suppose that $N\in\FIsoc (X)$ is isoclinic of slope $s$ at each point of $X$. Then the $i$-th cohomology of the complex \eqref{e:Cocone} vanishes for all $i<-s$.
\end{prop}

Proposition~\ref{p:0coh-vanishing} is equivalent to part (a) of the following one.

\begin{prop}  \label{p:coh-vanishing}
Let $K\supset\BQ_p$ be a finite extension and $\gamma\in K^\times$. Suppose that $N\in\FIsoc (X)\otimes_{\BQ_p}K$ is unit-root. Then 

(a) the $i$-th cohomology of the complex
 \begin{equation}  \label{e:2Cocone}
\Cocone (R\Gamma_{\crys} (X,N)\overset{F-\gamma}\longrightarrow R\Gamma_{\crys} (X,N))
\end{equation}
vanishes for all $i<v(\gamma )$;

(b) if $X$ is affine then the $i$-th cohomology of \eqref{e:2Cocone} also vanishes for all $i>v(\gamma )+1$;

(c) if $v(\gamma )<0$ then all cohomology groups of \eqref{e:2Cocone} vanish;

(d) if $v(\gamma )>\dim X$ then all cohomology groups of \eqref{e:2Cocone} vanish.

\end{prop}

(As usual, $\FIsoc (X)\otimes_{\BQ_p}K$ denotes the category of objects of $\FIsoc (X)$ equipped with $K$-action.)

If $v(\gamma )$ is a non-negative integer then Proposition~\ref{p:coh-vanishing} just says that the complex \eqref{e:2Cocone} is concentrated in degrees $v(\gamma )$ and $v(\gamma )+1$; for a more precise statement, see Proposition~\ref{p:essentially single degree} below.

Let us note that if $v(\gamma )>\dim X +1$ then Proposition~\ref{p:coh-vanishing}(d) immediately follows from Proposition~\ref{p:coh-vanishing}(a).

\subsection{Proof of Proposition~\ref{p:coh-vanishing}}
It suffices to prove Proposition~\ref{p:coh-vanishing} if $X$ is affine. From now on we assume this.

\subsubsection{A concrete realization of $R\Gamma_{\crys} (X,N)$}   \label{sss:concrete}

We fix a pair $(\sX ,\phi )$, where $\sX$ is a smooth formal scheme over the Witt ring $W(k)$ with special fiber $X$ and $\phi :\sX\to\sX$ is a lift of the absolute Frobenius of $X$. Let $\sX_n$ denote the reduction of $\sX$ modulo~$p^n$.

By \cite[Thm.~2.1]{Cr}, a unit-root object $N\in\FIsoc (X)\otimes_{\BQ_p}K$ is ``the same as" a lisse $K$-sheaf $\cN$  on $\sX_{\et}$ (i.e., a $\BQ_p$-sheaf equipped with an action of $K$). Let $\fo_K\subset K$ be the ring of integers. 
We have $\cN=\cN_0\otimes_{\fo_K}K$ for some torsion-free lisse $\fo_K$-sheaf $\cN_0$ on 
$X_{\et}$. 

Tensoring $\cN_0/p^n\cN_0$ by the structure sheaf $\cO_{\sX_n}$ (viewed as a sheaf on $X_{\et}$), one gets a vector bundle $L_n$ on 
$\sX_n$. The vector bundles $L_n$ on $\sX_n$ define a vector bundle $L$ on $\sX$ equipped with an integrable connection $\nabla$, an action of $\fo_K$, and an action of $\phi$ (i.e., a $\phi$-linear endomorphism of $H^0(X,N)$).

Let $C^{\bullet}$ denote the de Rham complex of $(L,\nabla )$. This is a complex of topologically free $\fo_K$-modules equipped with an endomorphism $F$ (the latter comes from the action of $\phi$). The complex $C^{\bullet}\otimes\BQ$ is a concrete realization of $R\Gamma_{\crys} (X,N)$.

\subsubsection{Lemmas in the spirit of Berthelot-Ogus}
The terms of the complex $C^{\bullet}$ are denoted by $C^j$. Let $\tilde C^{\bullet}\subset C^{\bullet}\otimes\BQ$ denote the subcomplex whose $j$-th term equals $\tilde C^j:=p^{-j}\cdot C^j\subset C^j\otimes\BQ$. It is clear that the morphism 
$F:C^{\bullet}\otimes\BQ\to C^{\bullet}\otimes\BQ$ maps $\tilde C^{\bullet}$ to $C^{\bullet}$.

The following lemmas and their proofs date back to  \cite{BO} (see \cite[Lemma~1.4]{BO} and  \cite[Props~1.5-1.7]{BO}).

\begin{lem}  \label{l:BO1}
(i) The complex $\tilde C^{\bullet}/p\tilde C^{\bullet}$ has zero differential.

(ii) The morphism $F:\tilde C^{\bullet}\to C^{\bullet}$ is a quasi-isomorphism.
\end{lem}

\begin{proof}
Statement (i) is clear. The terms of the complexes $\tilde C^{\bullet}$ and $C^{\bullet}$ are topologically free. So to prove~(ii), it suffices to check that the morphism $\tilde C^{\bullet}/p\tilde C^{\bullet}\to C^{\bullet}/pC^{\bullet}$ induced by $F$ is a quasi-isomorphism. This is a well-known interpretation of the inverse of the Cartier isomorphism due to Mazur \cite{Maz}.
\end{proof}

\begin{lem}   \label{l:BO2}
Let $\gamma\in K$, $i\in\BZ$, $i< v(\gamma )$. Let $C^{\bullet}_{\le i}$ (resp.~$\tilde C^{\bullet}_{\le i}$) denote the complex obtained from 
$C^{\bullet}$ (resp.~$\tilde C^{\bullet}$) by replacing the terms of degree $>i$ with zeros. Then
   
  (i) the morphism $F-\gamma :\tilde C^{\bullet}_{\le i}\otimes\BQ\to C^{\bullet}_{\le i}\otimes\BQ$ maps $\tilde C^{\bullet}_{\le i}$ to $C^{\bullet}_{\le i}$;
  
  (ii) the cohomology of the complex
  $\Cocone (\tilde C^{\bullet}_{\le i}\overset{F-\gamma}\longrightarrow C^{\bullet}_{\le i})$ is concentrated in degree $i+1$; this cohomology is zero if $i\ge\dim X$.
\end{lem}

\begin{proof}
Since $v(\gamma )>i$ we have
\begin{equation}   \label{e:zero mod m}
\gamma (\tilde C^{\bullet}_{\le i})\subset m_K\cdot C^{\bullet}_{\le i}\,,
\end{equation}
where $m_K$ is the maximal ideal of $\fo_K$. Statement (i) follows from \eqref{e:zero mod m} and the inclusion $F(\tilde C^{\bullet}_{\le i})\subset C^{\bullet}_{\le i}$.

$\Cocone (\tilde C^{\bullet}_{\le i}\overset{F-\gamma}\longrightarrow C^{\bullet}_{\le i})$ is a complex of topologically free $\fo_K$-modules. By Lemma~\ref{l:BO1} and formula~\eqref{e:zero mod m}, the cohomology of its reduction modulo $m_K$ is concentrated in degree $i+1$, and if $i\ge\dim X$ this cohomology is zero. Statement (ii) follows. \end{proof}

\subsubsection{End of the proof}
Let $\gamma\in K$, $i\in\BZ$, $i< v(\gamma )$. By Lemma~\ref{l:BO2}(ii), the cohomology of the complex
$\Cocone (\tilde C^{\bullet}_{\le i}\otimes\BQ\overset{F-\gamma}\longrightarrow C^{\bullet}_{\le i}\otimes\BQ)$ is concentrated in degree $i+1$. So the cohomology of the complex $\Cocone (\tilde C^{\bullet}\otimes\BQ\overset{F-\gamma}\longrightarrow C^{\bullet}\otimes\BQ)$ is concentrated in degrees $>i$. This is equivalent to Proposition~\ref{p:coh-vanishing}(a).

Now suppose that $i> v(\gamma )+1$ and $j\ge i-1$. Then $\gamma^{-1}F(C^j)\subset m_K\cdot C^j$. So the operator $1-\gamma^{-1}F:C^j\to C^j$ is invertible: its inverse equals $1+\gamma^{-1}F+(\gamma^{-1}F)^2+\ldots \;$. So 
$F-\gamma:C^j\otimes\BQ\to C^j\otimes\BQ$ is invertible for all $j\ge i-1$. Therefore $\Cocone (C^{\bullet}\otimes\BQ\overset{F-\gamma}\longrightarrow C^{\bullet}\otimes\BQ)$ is quasi-isomorphic to the complex $\Cocone (C^{\bullet}_{<i-1}\otimes\BQ\overset{F-\gamma}\longrightarrow C^{\bullet}_{<i-1}\otimes\BQ)$. The latter complex is concentrated in degrees $<i$, and if $i=1$ the complex is zero. This implies Proposition~\ref{p:coh-vanishing}(b-c).

Proposition~\ref{p:coh-vanishing}(d) follows from the second part of Lemma~\ref{l:BO2}(ii) applied for $i=\dim X$.
\qed

The next subsection is not used in the rest of the article.

\subsection{A refinement of Proposition~\ref{p:coh-vanishing}}   \label{ss:refinement}
If $v(\gamma )$ is a non-negative integer then Proposition~\ref{p:coh-vanishing} says that the complex \eqref{e:2Cocone} is concentrated in degrees $v(\gamma )$ and $v(\gamma )+1$. Here is a more precise statement, whose proof given below was explained to us by L.~Illusie and K.~Kato.

\begin{prop} \label{p:essentially single degree}
Let $r\in\BZ$, $r\ge 0$. Suppose that in the situation of Proposition~\ref{p:coh-vanishing} one has $v(\gamma )=r$. Then there exists a projective system 
\[
\ldots\to \cG_3\to \cG_2\to \cG_1
\]
of sheaves of abelian groups on $X_{\et}$ such that

(i) each $\cG_n$ is a flat sheaf of $(\BZ/p^n\BZ)$-modules, and the morphism $\cG_{n+1}\to \cG_n$ identifies $\cG_n$ with $\cG_{n+1}/p^n \cG_{n+1}$;

(ii) if $X'$ is any scheme etale over $X$ and $N'\in\FIsoc (X')\otimes_{\BQ_p}K$ is the pullback of $N$ then one has a canonical isomorphism
\begin{equation}  \label{e:3Cocone}
\Cocone (R\Gamma_{\crys} (X',N')\overset{F-\gamma}\longrightarrow R\Gamma_{\crys} (X',N'))\simeq\underset{n}{\underset{\longleftarrow}\lim} R\Gamma (X'_{\et}, \cG_n)[-r]\otimes_{\BZ_p}\BQ_p \, .
\end{equation}

\end{prop}

We will see that the sheaves $\cG_n$ are not constructible if $r>0$.

\begin{rem}   \label{r:gamma=p^r}
It suffices to prove Proposition~\ref{p:essentially single degree} if $\gamma=p^r$ (otherwise twist $N$ by a suitable unit-root object of $\FIsoc (\Spec k)\otimes_{\BQ_p}K$).
\end{rem}

\subsubsection{Constructing the sheaves $\cG_n$}
We will assume that $\gamma =p^r$ (see Remark~\ref{r:gamma=p^r}). Under this assumption, the sheaves $\cG_n$ from 
Proposition~\ref{p:essentially single degree} are constructed as follows. 

On $X_{\et}$ we have the projective system of de Rham-Witt complexes $W_n\Omega^{\bullet}_X$. For each $n$ and $r$ we have the ``logarithmic" subsheaf
$W_n\Omega^r_{X,\log}\subset W_n\Omega^r_X$ defined in \S 5.7 of Ch. I of Illusie's article \cite{Il} (p.~596-597). This is a sheaf of $(\BZ/p^n\BZ)$-modules.
For fixed $r$ and variable $n$ the sheaves $W_n\Omega^r_{X,\log}$ form a projective system. It is known that the sheaf $W_n\Omega^r_{X,\log}$ is flat over
$\BZ/p^n\BZ$ and the morphism $W_{n+1}\Omega^r_{X,\log}\to W_n\Omega^r_{X,\log}$ identifies $W_n\Omega^r_{X,\log}$ with $W_{n+1}\Omega^r_{X,\log}/p^n W_{n+1}\Omega^r_{X,\log}$ (see Lemma 3 on p.~779 of \cite{CSS}). 

Let us note that $W_n\Omega^r_{X,\log}$ has a $K$-theoretic description: by Theorem~5.1 of \cite{Mor} (which is due to many authors), $W_n\Omega^r_{X,\log}$ identifies with the sheaf $K_r/p^nK_r$ on $X_{\et}$ and also with $K^M_r/p^nK^M_r$, where $K^M_r$ is the sheaf of  
Milnor  $K$-groups\footnote{Theorem~5.1 of \cite{Mor} involves the ``improved" Milnor $K$-groups rather than the usual ones. However, for local rings with \emph{infinite} residue fields the ``improved" Milnor $K$-groups are equal to the usual ones, and the residue field of each stalk of $\cO_{X_{\et}}$ is infinite (because it is separably closed).}. For instance, $W_n\Omega^0_{X,\log}$ is the constant sheaf with fiber $\BZ/p^n\BZ$, and $W_n\Omega^1_{X,\log}=\cO_{X_{\et}}^\times/(\cO_{X_{\et}}^\times)^{p^n}$.

Now set
\begin{equation}   \label{e:defining G_n}
\cG_n:=\cN_0\otimes_{\BZ_p}W_n\Omega^r_{X,\log}\; ,
\end{equation}
where $\cN_0$ is as in \S\ref{sss:concrete}.

\subsubsection{Constructing the isomorphism~\eqref{e:3Cocone}}
As before, we assume that $\gamma=p^r$ and $\cG_n$ is defined by \eqref{e:defining G_n}. We will also assume that the scheme $X'$ from Proposition~\ref{p:essentially single degree}(ii) equals $X$.

According to \cite{Il}, the de Rham complex $C^{\bullet}$ introduced in \S\ref{sss:concrete} is canonically quasi-isomorphic to
$\BQ\otimes\underset{n}{\underset{\longleftarrow}\lim} R\Gamma (X_{\et},\cN_0\otimes_{\BZ_p}W_n\Omega^{\bullet}_X)$. So the problem is to compute the complex
\begin{equation}   \label{e:complex to compute}
\BQ\otimes\underset{n}{\underset{\longleftarrow}\lim} R\Gamma (X_{\et},\cN_0\otimes_{\BZ_p}\Cocone (W_n\Omega^{\bullet}_X\overset{F-p^r}\longrightarrow W_n\Omega^{\bullet}_X )).
\end{equation}

In addition to $F$, we have the ``de Rham-Witt" operators $\sfV$ and $\sfF$ satisfying $\sfV\sfF=\sfF\sfV=p$ (see~\cite{Il}). Unlike $F$, they do \emph{not} commute with the differential $d$. Note that while $\sfV$ is an endomorphism of $W_n\Omega^i_X$, the operator $\sfF$ is a morphism from 
$W_{n+1}\Omega^i_X$ to its quotient $W_n\Omega^i_X$. However, the operator $p\sfF: W_n\Omega^i_X\to W_n\Omega^i_X$ is well-defined and nilpotent. Moreover, if $i=0$ then $\sfF: W_n\Omega^i_X\to W_n\Omega^i_X$ is well-defined. Recall that the morphism $F:W_n\Omega^i_X\to W_n\Omega^i_X$ equals $p^i\sfF$.

\begin{lem}
If $i\ne r$ then the kernel and cokernel of $F-p^r: W_n\Omega^i_X\to W_n\Omega^i_X$ are killed by a power of $p$ independent of $n$.
\end{lem}

\begin{proof}
If $i> r$ write $F-p^r=p^r (p^{i-r}\cdot \sfF-1)$ and note that $p^{i-r}\cdot \sfF-1$ is invertible because $p^{i-r}\cdot \sfF$ is nilpotent.

If $i< r$ write $F-p^r=F(1-p^{r-i-1} \sfV)$. The operator $1-p^{r-i-1} \sfV$ is invertible because $ \sfV$ is nilpotent. Finally, the kernel and cokernel of 
$F:W_n\Omega^i_X\to W_n\Omega^i_X$ are killed by $p^{i+1}$ because $F\sfV=\sfV F=p^{i+1}$.
\end{proof}

The lemma implies that the complex \eqref{e:complex to compute} is canonically isomorphic to
\[
\BQ\otimes\underset{n}{\underset{\longleftarrow}\lim} R\Gamma (X_{\et},\cN_0\otimes_{\BZ_p}\Cocone (W_n\Omega^r_X\overset{F-p^r}\longrightarrow W_n\Omega^r_X ))[-r]
\]
and therefore to
\[
\BQ\otimes\underset{n}{\underset{\longleftarrow}\lim} R\Gamma (X_{\et},\cN_0\otimes_{\BZ_p}\Cocone (W_n\Omega^r_X\overset{\sfF-1}\longrightarrow W_n\Omega^r_X/\sfF (A_n^r ))[-r],
\]
where $A_n^r:=\Ker (W_{n+1}\Omega^r_X\epi W_n\Omega^r_X )$. Finally, Lemma 2 on p.~779 of \cite{CSS} tells us\footnote{To see this, note that the sheaf
$d\sfV^{n-1}W_n\Omega^{r-1}_X$ from Lemma 2 on p.~779 of \cite{CSS} is equal to $\sfF (A_n^r)$. This follows from the formula $A_n^r=\sfV^n W_{n+1}\Omega^r_X+d\sfV^n W_{n+1}\Omega^r_X$ (see Proposition~3.2 on p.568 of \cite{Il}).} that 
\[
\Cocone (W_n\Omega^r_X\overset{\sfF-1}\longrightarrow W_n\Omega^r_X/\sfF (A_n^r ))=W_n\Omega^r_{X,\log}\, .
\]
Thus we get the desired isomorphism \eqref{e:3Cocone}. This finishes the proof of Proposition~\ref{p:essentially single degree}.

\begin{ex}
If $X=(\BG_m)^n$, $N$ is the unit object of the tensor category $\FIsoc (X)$, and $0\le r<n$ 
then a direct computation shows that both $H^r$ and $H^{r+1}$ of the complex
\[
\Cocone (R\Gamma_{\crys} (X,N)\overset{F-p^r}\longrightarrow R\Gamma_{\crys} (X,N))
\]
are nonzero, and $H^{r+1}$ has infinite dimension over $\BQ_p$.
\end{ex}

\section{Some reformulations of Theorem~\ref{t:main}} \label{s:reformulations}
\subsection{The canonical decomposition corresponding to a big gap between the slopes}
Theorem~\ref{t:main} is clearly equivalent to the following
\begin{prop}  \label{p:reform1}
Let $M\in\FIsoc (X)$. Suppose that 
$a_i^\eta (M )-a_{i+1}^\eta (M )>1$ for some $i$. Then 
$M$ admits a decomposition 
\begin{equation}  \label{e:gap-decomp}
M=M_1\oplus M_2
\end{equation}
 such that the slopes of $M_1$ (resp.~$M_2$) at $\eta$ are the numbers $a_j^\eta (M )$ for $j\le i$ (resp. $j>i$). \qed
\end{prop}

\begin{rem}
It is clear that in the  situation of Proposition~\ref{p:reform1} one has
\begin{equation}  \label{e:orthog}
\Hom (M_1,M_2)=\Hom (M_2,M_1)=0.
\end{equation}
This implies that the decomposition \eqref{e:gap-decomp} is unique. 
\end{rem}

\begin{prop}  \label{p:reform2}
Let $M\in\FIsoc (X)$. Suppose that $\End M$ has no non-trivial central idempotents. Then $a_i^\eta (M )-a_{i+1}^\eta (M )\le 1$ for all $i\in\{1,\ldots ,n-1\}$, where $n$ is the rank of $M$.
\end{prop}

\begin{proof}
Suppose that $a_i^\eta (M )-a_{i+1}^\eta (M )>1$ for some $i$. Consider the corresponding decomposition~\eqref{e:gap-decomp}. Let $\pi\in\End M$ be the projection to $M_1$. Then $\pi$ is a non-trivial idempotent. It is central by \eqref{e:orthog}.
\end{proof}

\subsection{The categories $\FIsoc (X)\otimes_{\BQ_p}K$ and $\FIsocd (X)\otimes_{\BQ_p}\BQbar_p$}
\subsubsection{The definitions}
For $K$ a finite extension of $\BQ_p$, we define $\FIsoc (X)\otimes_{\BQ_p}K$ to be the category of objects of $\FIsoc (X)$ equipped with an action of $K$; define $\FIsocd (X)\otimes_{\BQ_p}K$ similarly. Set
$$\FIsoc (X)\otimes_{\BQ_p}\BQbar_p:=\underset{K}{\underset{\longrightarrow}\lim} \FIsoc (X)\otimes_{\BQ_p} K, \quad
\FIsocd (X)\otimes_{\BQ_p}\BQbar_p:=\underset{K}{\underset{\longrightarrow}\lim} \FIsocd (X)\otimes_{\BQ_p} K,$$ where $K$ runs through the set of all subfields of $\BQbar_p$ finite over $\BQ_p$. 
For an object $M$ in $\FIsoc (X)\otimes_{\BQ_p}K$ or in $\FIsocd (X)\otimes_{\BQ_p}\BQbar_p$, the slopes $a_i^x(M)$, $x\in X$, are defined similarly to \S\ref{sss:Slopes}.

\begin{rem}   \label{r:2Forg}
For $K$ a finite extension of $\BQ_p$, one has the forgetful functor
$$\Forg_{K/\BQ_p}:\FIsoc (X)\otimes_{\BQ_p}K\to\FIsoc (X).$$
If $M\in \FIsoc (X)$ has rank $n$ then $\Forg_{K/\BQ_p}(M)$ has rank $dn$, where $d:=[K:\BQ_p]$. The set of slopes of $\Forg_{K/\BQ_p}(M)$ is equal to that of $M$, but the multiplicity of each slope is multiplied by $d$.
\end{rem}

\begin{prop}   \label{p:remain valid}
Theorem~\ref{t:main} and Propositions~\ref{p:reform1}, \ref{p:reform2} remain valid for 
$\FIsoc (X)\otimes_{\BQ_p}K$, where $K$ is a finite extension of $\BQ_p$.
\end{prop}

\begin{proof}
By Remark~\ref{r:2Forg}, Proposition~\ref{p:reform2} for $M\in\FIsoc (X)\otimes_{\BQ_p}K$ follows from Proposition~\ref{p:reform2} for 
$\Forg_{K/\BQ_p}(M)\in\FIsoc (X)$. Theorem~\ref{t:main} and Proposition~\ref{p:reform1} for 
$\FIsoc (X)\otimes_{\BQ_p}K$ follow from Proposition~\ref{p:reform2} for $M\in\FIsoc (X)\otimes_{\BQ_p}K$.
\end{proof}

\begin{cor}  \label{c:remain valid}
Theorem~\ref{t:main} and Propositions~\ref{p:reform1}, \ref{p:reform2} remain valid for 
$\FIsoc (X)\otimes_{\BQ_p}\BQbar_p$. \qed
\end{cor}

\section{Proof of Theorem~\ref{t:l-adic}} \label{s:proof of l-adic}
Recall that $\BQbar$ denotes the algebraic closure of $\BQ$ in $\BQbar_{\ell}$. We fix a valuation $v:\BQbar^\times\to\BQ$ such that $v(p)=1$; slopes of algebraic $\BQbar_{\ell}$-sheaves are defined using $v$.

\subsection{Proof of Theorem~\ref{t:l-adic} for irreducible sheaves on curves}   \label{ss:irreducible on curves}
For any subfield $E\subset\BQbar$, let $E_v$ denote the completion of $E$ with respect to $v$. The union of the fields $E_v$ corresponding to all subfields $E\subset\BQbar$ finite over $\BQ$ is an algebraic closure of $\BQ_p$; we denote it by $\BQbar_p$.

Assume that $\dim X=1$ and $\E$ is irreducible. Then Theorem 4.4.1 of Abe's work \cite{A} provides an irreducible 
object $M\in \FIsocd (X)\otimes_{\BQ_p}\BQbar_p$ of rank $n$ such that for every $x\in |X|$, the characteristic polynomials of the geometric Frobenius acting on $M_{\bar x}$ and $\E_{\bar x}$ are equal to each other. 
Then the multiset of slopes of $\E$ at any $x\in |X|$ is equal to that of $M$. Define the numbers $a_i^\eta (\E )$ to be the slopes of $M$ at the generic point $\eta\in X$. Applying semicontinuity of the Newton polygon to $M$, we see that the numbers $a_i^\eta (\E )$ satisfy properties (i)--(ii) from Theorem~\ref{t:l-adic}.

To check (iii), recall that the functor $\FIsocd (X)\to\FIsoc (X)$ is fully faithful, so by Lemma~\ref{l:indecomposability}, $M$ is indecomposable as an object of $\FIsoc (X) \otimes_{\BQ_p}\BQbar_p$. It remains to apply Theorem~\ref{t:main} to $M$, which is possible by Corollary~\ref{c:remain valid}. \qed

\subsection{Proof of Theorem~\ref{t:l-adic} for  curves}
In \S\ref{ss:irreducible on curves} we proved Theorem~\ref{t:l-adic} assuming that $\E$ is irreducible. This immediately implies 
Theorem~\ref{t:l-adic}(i-ii) for any $\E$. Theorem~\ref{t:l-adic}(iii) for any indecomposable $\E$ will be proved in \S\ref{sss:indecomposable on curves}.

\subsubsection{A result of Deligne} As before, slopes are defined using a fixed valuation $v:\BQbar^\times\to\BQ$.

\begin{prop}   \label{p:slope estimates}
Let $X$ be a scheme of finite type over $\BF_p$; let $\pi$ denote the morphism $X\to\Spec\BF_p$.
Let $\E$ be a lisse $\BQbar_{\ell}$-sheaf on $X$.

(i) Suppose that $\E$ is algebraic in the sense of \S\ref{sss:algebraicity} and the slopes of $\E$ (with respect to some $p$-adic place of $\BQbar$) are in the interval $[r,s]$. Then for each $i$ the sheaf $R^i\pi_!\E$ is algebraic, and its slopes are in the interval 
\begin{equation}   \label{e:new interval}
[r+\max\{0,i-n\},s+\min\{i,n\}], \quad \mbox{where }n:=\dim X.
\end{equation}

(ii) If $X$ is smooth then the same statements hold for $R^i\pi_*\E$.
\end{prop}

\begin{proof}
Statement (i) is a reformulation of Theorems 5.2.2 and 5.4 of \cite[expos\'e~XXI]{SGA7}, which are due to  Deligne. 
Statement (ii) follows by Verdier duality.
\end{proof}

\begin{ex}   
Let $X$ be the $n$-th power of a non-supersingular elliptic curve. Let $\E =(\BQbar_{\ell})_X$ and $r=s=0$. Then all integers from the interval 
\eqref{e:new interval} appear as slopes of $R^i\pi_*\E$.
\end{ex}

\begin{rem}
The slope estimates from Proposition~\ref{p:slope estimates} remain valid for overconvergent $F$-isocrystals and rigid cohomology instead of lisse $\BQbar_{\ell}$-sheaves and $\ell$-adic cohomology, see \cite[Thm.~5.4.1]{Ke3}.
\end{rem}

\begin{cor} \label{c:slope estimates}
Let $X$ be an irreducible smooth variety over $\BF_p$. Let $\E_1 ,\E_2$ be algebraic lisse $\BQbar_{\ell}$-sheaves on $X$. Suppose that one of the following assumptions holds:

(i) for some non-empty open $U\subset X$ all the slopes of $\E_1^*\otimes\E_2$ at all points of $|U|$ are $<-1$;

(ii) for some non-empty open $U\subset X$ all the slopes of $\E_1^*\otimes\E_2$ at all points of $|U|$ are $>0$.

\noindent Then $\Ext^1 (\E_1 ,\E_2 )=0$. 
\end{cor}

\begin{proof}
Let $j:U\mono X$ be the embedding.  Let $\pi$ be the morphism $X\to\Spec\BF_p$.

 Since $X$ is normal, the map $\Ext^1 (\E_1 ,\E_2 )\to\Ext^1 (j^*\E_1 ,j^*E_2 )$ is injective. So we can assume that $U=X$. Then  
all the slopes of the sheaves
$$R^i\pi_*(\E_1^*\otimes\E_2), \quad i\in\{ 0,1\}$$
are non-zero  by Proposition~\ref{p:slope estimates}(ii). So $\Ext^1 (\E_1 ,\E_2 )=H^1(\Spec\BF_p ,R\pi_*(\E_1^*\otimes\E_2))=0$.
\end{proof}

\subsubsection{Proof of Theorem~\ref{t:l-adic}(iii) if $\dim X=1$}   \label{sss:indecomposable on curves}
The case where $\E$ is irreducible was treated in \S\ref{ss:irreducible on curves}. The case where $\E$ is indecomposable but not necessarily irreducible follows by Corollary~\ref{c:slope estimates}. \qed

\subsection{The case $\dim X>1$.}   \label{ss:arbitrary} 
\begin{lem} \label{l:bounded denom}
(i) There exists $N\in\BN$ such that  $a_i^x(\E )\in N^{-1}\BZ$ for all $x\in|X|$ and $i\in\{1,\ldots ,n\}$.

(ii) Suppose that $\E$ is irreducible and the rank 1 local system $\det\E$ has finite order. Let $E\subset\BQbar_\ell$ be a subfield finite over $\BQ$ such that for every $x\in |X|$ all the coefficients of the characteristic polynomial of the geometric Frobenius acting on $\E_{\bar x}$ belong to $E$ (such $E$ exists by  \cite[Thm.~3.1]{De2}). Then the 
number $N$ from statement (i) can be taken to be $n!\cdot [E_v:\BQ_p]$, where $E_v$ is the completion of $E$ with respect to the valuation $v:\BQbar^\times\to\BQ$ that was fixed at the beginning of \S\ref{s:proof of l-adic}.
\end{lem}

\begin{proof}
It suffices to prove statement (ii). Let $x\in |X|$. Using \cite[Prop.~2.17]{Dr} (or the Bertini argument from \cite[1.7-1.8]{De2}), one finds a smooth connected curve $C$ over $\BF_p$ and a morphism $f:C\to X$ such that $x\in f(C)$ and $f^*\E$ is irreducible.  Theorem 4.4.1 of Abe's work \cite{A} provides an irreducible 
object $M\in \FIsocd (C)\otimes_{\BQ_p}\BQbar_p$ of rank $n$ such that for every $c\in |C|$, the characteristic polynomials of the geometric Frobenius acting on $M_{\bar c}$ and $(f^*\E)_{\bar c}$ are equal to each other. It remains to show that all slopes of $M$ at each point of $C$ are in $(n!\cdot [E_v:\BQ_p])^{-1}\BZ$.

Note that the isomorphism class of $M$ is $\Gal (\BQbar_p/E_v)$-invariant. Because of the Brauer obstruction, this does not imply that $M\in \FIsocd (C)\otimes_{\BQ_p}E_v$; the obstruction is an element $\beta\in\Br (E_v)$. Note that $\det M$ has finite order and therefore slope $0$; so by \cite{Cr} we can think of $\det M$ as a representation of the fundamental group of $X$. Therefore $\det M\in \FIsocd (C)\otimes_{\BQ_p}E_v$. So $n\beta=0$. So if $L$ is a degree $n$ extension of $E_v$ then $\beta$ has zero image in 
$\Br (L)$, which means that $M\in  \FIsocd (C)\otimes_{\BQ_p}L$. Therefore for any $c\in |C|$ one has $M_c\in  \FIsoc (c)\otimes_{\BQ_p}L$.

Consider the slope decomposition $M_c=\bigoplus\limits_{r\in\BQ}(M_c)_r$. Let $\Forg_{L/\BQ_p}:\FIsoc (c)\otimes_{\BQ_p}L\to\FIsoc (c)$ be the forgetful functor from Remark~\ref{r:2Forg}. Then $\det \Forg_{L/\BQ_p}((M_c)_r)$ is a rank 1 object of $\FIsoc (c)$, so its slope is in $\BZ$. But this slope equals $rn\cdot [E_v:\BQ_p]\cdot\dim (M_c)_r$, so 
$rn\cdot [E_v:\BQ_p]\cdot\dim (M_c)_r\in\BZ$. If $0<\dim (M_c)_r<n$ this implies that $rn!\cdot [E_v:\BQ_p]\in\BZ$. If $\dim (M_c)_r=n$ then $\det (M_c)_r=\det M_c$ has slope~$0$, so $r=0$ and again $rn!\cdot [E_v:\BQ_p]\in\BZ$.
\end{proof}

\begin{lem}  \label{l:irreducibility}
For every $x\in |X|$ there exists a smooth connected curve $C$ over $\BF_p$ and a morphism $f:C\to X$ such that $x\in f(C)$ and $f^*\E$ is indecomposable.
\end{lem}

\begin{proof}
This is a consequence of \cite[Prop.~2.17]{Dr}.
\end{proof}

\begin{lem}    \label{l:boundedness}
For each $r\in \{ 1,\ldots , n\}$ the function
\begin{equation}   \label{e:sum to r}
 x\mapsto \sum\limits_{i=1}^r a_i^x(\E ), \quad x\in |X|
 \end{equation}
is bounded above.
\end{lem}

\begin{proof}
Without loss of generality, one can assume that $\E$ is indecomposable. By \cite[Prop.~1.3.4(i)]{De}, after tensoring $\E$ by a rank 1 local system on $\Spec\BF_p$ one can also assume that $(\det\E)^{\otimes m}$ is trivial for some $m$, so $\sum\limits_{i=1}^n a_i^x(\E )=0$. In this situation $\sum\limits_{i=1}^r a_i^x(\E )\le r(n-r)/2$ by Lemma~\ref{l:irreducibility} and Remark~\ref{r:weak estimate} (the latter is applicable because we already proved  Theorem~\ref{t:l-adic} for curves).
\end{proof}

\begin{lem}    \label{l:2irreducibility}
Let $\cT$ denote the following topology on $|X|$: a subset $F\subset |X|$ is $\cT$-closed if and only if $F\cap |C|$ is closed for all curves $C\subset X$. Then $|X|$ is irreducible with respect to $\cT$. \end{lem}

\begin{proof}
Suppose that $F_1,F_2\subset |X|$ are $\cT$-closed and different from $|X|$. Choose $x_1,x_2\in|X|$ so that $x_i\not\in F_i$. By Hilbert irreducibility (e.g., by \cite[Thm.~2.15(i)]{Dr}), there exists an irreducible curve $C\subset X$ containing $x_1$ and $x_2$. Then the sets $|C|\cap F_i$ are finite., so $|C|\not\subset F_1\cup F_2$. Therefore $F_1\cup F_2\ne |X|$.
\end{proof}

\begin{proof}[Proof of Theorem~\ref{t:l-adic}]
By Lemmas~\ref{l:bounded denom} and \ref{l:boundedness}, the function \eqref{e:sum to r} has a maximal value $s_r$; let $U_r$ denote the set of all $x\in |X|$ for which this value is attained, and let $U:=\bigcap\limits_r U_r$. Define the numbers $a_i^\eta (\E )$ as follows: $a_1^\eta (\E ):=s_1$,  $a_i^\eta (\E ):=s_i-s_{i-1}$ for $i>1$. 

Let us prove that the numbers $a_i^\eta (\E )$ have the properties stated in Theorem~\ref{t:l-adic}. Each $U_r$ is clearly non-empty, and 
by \S\ref{sss:indecomposable on curves} it is $\cT$-open. So $U$ is non-empty and $\cT$-open by Lemma~\ref{l:2irreducibility}. 
For every $i<n$ one has $a_i^\eta (\E )\ge a_{i+1}^\eta (\E )$ because $a_j^\eta (\E )=a_j^x (\E )$ for all $x\in U$ and all $j$.
By construction, $s_r^x (\E )\le s_r^\eta (\E )$ for all $r$. Finally, if $\E$ is indecomposable then for $x\in U$ one has $a_i^x (\E )- a_{i+1}^x (\E )\le 1$: to see this, apply \S\ref{sss:indecomposable on curves} to the curve $C$ from Lemma~\ref{l:irreducibility}.

The uniqueness part of Theorem~\ref{t:l-adic}(i) follows from Lemma~\ref{l:2irreducibility}.
\end{proof}

\section{Slopes for automorphic representations of $GL(n)$}   \label{s:automorphic}
\subsection{Definition of slopes}  \label{ss:automorphic slopes}
\subsubsection{Some notation}   \label{sss:automorphic notation}
Suppose that $X$ is a smooth irreducible curve over $\BF_p$ (it is \emph{not} assumed to be projective).
The order of the residue field of $x\in |X|$ will be denoted by $q_x$. Let $F$ denote the field of rational functions on $X$ and $\BA_F$ its adele ring. Let $F_x$ denote the completion of $F$ at
$x\in|X|$ and $O_x\subset F_x$ its ring of integers.

We fix an algebraic closure  $\BQbar_p\supset\BQ_p$. Let $\BZbar_p$ denote the ring of integers of $\BQbar_p$. 
Let $v:\BQbar_p^\times\to\BQ$ denote the $p$-adic valuation normalized so that $v(p)=1$. For each $x\in |X|$, let $v_x:\BQbar_p^\times\to\BQ$ denote the $p$-adic valuation normalized so that $v_x(q_x)=1$ (so $v_x=v/\deg x$, where $\deg x:=\log_pq_x$).

\subsubsection{$p^{1/2}$ and the Satake parameter} \label{sss:sqrt}
We fix a square root of $p$ in $\BQbar_p$ and denote it by $p^{1/2}$.
For each $x\in |X|$ we set $q_x^{1/2}:=(p^{1/2})^{\deg x}$; we use this square root of $q_x$ in the definition of the Satake parameter of an unramified irreducible representation of $GL(n,F_v)$ over $\BQbar_p$.

\subsubsection{Slopes}   \label{sss:automorphic slopes}
Let $\pi$ be an irreducible admissible representation of $GL(n,\BA_F)$ over $\BQbar_p$. Let $S_\pi$ denote the set of all $x\in |X|$ such that $\pi_x$ is ramified.

For every $x\in |X|\setminus S_\pi$ the Satake parameter of $\pi_x$ 
is an $n$-tuple $(\gamma_1,\ldots ,\gamma_n )\in(\BQbar_p^\times)^n$ defined up to permutations. Set $a_i^x (\pi ):=v_x(\gamma_i )$. Note that $a_i^x (\pi )$ does not depend on the choice of $p^{1/2}$ in~\S\ref{sss:sqrt}. We order the numbers 
$a_i^x (\pi )\in\BQ$ so that $a_i^x (\pi )\ge a_{i+1}^x (\pi )$. These numbers will be called the \emph{slopes} of $\pi$ at $x$.

\subsubsection{Slopes and the central character} \label{sss:central character}
Let $\eta :\BA_F^{\times}\to\BQbar_p^{\times}$ denote the central character of $\pi$. Then for every $x\in |X|\setminus S_\pi$ and every 
$u\in F_x^{\times}$ one has
\begin{equation}   \label{e:central character}
v(\eta (u))=-v(|u|)\cdot\sum\limits_{i=1}^n a_i^x(\pi ),
\end{equation}
where $|u|\in\BQ^\times$ is the normalized absolute value. If $\eta$ is trivial on $F^\times\subset \BA_F^{\times}$ then there exists 
$c\in\BQ^\times$ such that $v(\eta (u))=c\cdot v(|u|)$ for all $u\in\BA_F^{\times}$. By \eqref{e:central character}, this implies that
the number $\sum\limits_{i=1}^n a_i^x(\pi )$ does not depend on $x\in |X|\setminus S_\pi$.

\subsection{The result}  \label{ss:automorphic result}
As before, let $\pi$ be an irreducible representation of $GL(n,\BA_F)$ over $\BQbar_p$. 

\begin{theorem}   \label{t:automorphic}
Suppose that $\pi$ is cuspidal automorphic. Then

(i) there exist rational numbers $a_i^\eta (\pi )$, $1\le i\le n$, such that $a_i^x (\pi )=a_i^\eta (\pi )$ for all but finitely many $x\in |X|$;

(ii) for all  $x\in |X|\setminus S_\pi$ and  $r\in \{ 1,\ldots ,n-1\}$ one has
\[
\sum\limits_{i=1}^r a_i^x(\pi )\le \sum\limits_{i=1}^r a_i^\eta (\pi );
\]

(iii) $a_i^\eta (\pi )-a_{i+1}^\eta (\pi )\le 1$ for all $i\in \{ 1,\ldots , n-1\}$.
\end{theorem}

\begin{proof}
By Theorem 4.2.2 of Abe's work \cite{A}, $\pi$ corresponds (in the sense of Langlands) to some irreducible object of $\FIsocd (X\setminus S_\pi )\otimes_{\BQ_p}\BQbar_p$.
So statements (i)--(ii) hold by semicontinuity of the Newton polygon, and (iii) holds by Theorem~\ref{t:main} and Corollary~\ref{c:remain valid}.
\end{proof}

\begin{rem}   \label{r:Vincent's inequality}
Let $\pi$ be as in Theorem~\ref{t:automorphic}. Set $A(\pi):=\frac{1}{n}\cdot\sum\limits_{i=1}^n a_i^x(\pi )$; by \S\ref{sss:central character}, this number does not depend on $x$. Similarly to Corollary~\ref{c:any x}, Theorem~\ref{t:automorphic} implies that
\begin{equation}   \label{e:3any x}
\sum\limits_{i=1}^r a_i^x(\pi )-rA(\pi )\le r(n-r)/2 \quad \mbox{ for all } x\in x\in |X|\setminus S_\pi \mbox{ and } r\in \{ 1,\ldots , n-1\}.
\end{equation}
On the other hand, V.~Lafforgue explained in \cite{VLa} that the inequality \eqref{e:3any x} has an easy direct proof. We will recall it  in \S\ref{sss:Vincent's proof}.
\end{rem}

\begin{quest}    \label{q:direct proof}
Can Theorem~\ref{t:automorphic} be proved directly (i.e., without passing to $F$-isocrystals)?
\end{quest}

As already mentioned, the inequality \eqref{e:3any x} has a direct proof. If $n=2$ then formula~\eqref{e:3any x} means that 
$a_1^x (\pi )-a_2^x (\pi )\le 1 \mbox{ for all }x\in |X|\setminus S_\pi\,.$
So for $n=2$ the only question is whether statements (i)--(ii) of Theorem~\ref{t:automorphic} have a direct proof.

\subsection{Reformulation in terms of Hecke eigenvalues}   \label{ss:Hecke reformulation}
\subsubsection{Hecke eigenvalues in terms of the Satake parameter} \label{sss:Hecke&Langlands}
Let $\pi$ be a smooth representation of $GL(n,\BA_F)$ over $\BQbar_p$. Then for each $x\in |X|\setminus S_\pi$ one has 
the usual Hecke operators $T_i^x$, $1\le i\le n$, acting on the subspace of $GL(n,O_x)$-invariants of $\pi$. If $\pi$ is irreducible and admissible then $T_i^x$ acts as multiplication by some number $t_i^x (\pi )\in\BQbar_p$. The numbers $t_i^x (\pi )$ (where $x\in |X|\setminus S_\pi$ and $1\le i\le n$) are called \emph{Hecke eigenvalues}. 
Formula~(3.14) of \cite{Gr} (which goes back to T.~Tamagawa) tells us that
\begin{equation}   \label{e:Hecke&Langlands}
1+\sum_{i=1}^n (q_x^{1/2})^{-i(n-i)}t_i^x (\pi )z^i=\prod_{i=1}^n (1+\gamma_iz), 
\end{equation}
where $(\gamma_1,\ldots ,\gamma_n )$ is the Satake parameter of $\pi_x$, and $z$ is a variable.

\begin{lem}   \label{l:Hecke&slopes}
For all $r\in\{ 0,\ldots ,n\}$ and $x\in |X|\setminus S_\pi$ one has 
\begin{equation}   \label{e:Newt}
\sum\limits_{i=1}^{n-r} a_i^x (\pi )=v_x(t^x_n(\pi ))-\Newt^x_\pi (r),
\end{equation}
where $\Newt^x_\pi :\{ 0,1,\ldots ,n\}\to\BQ$ is the biggest convex function such that 
\begin{equation} \label{e:2Newt}
\Newt^x_\pi (0)=0,\quad\quad \Newt^x_\pi (r)\le v_x(t^x_{r}(\pi ))+ \frac{r(r-n)}{2} \mbox{ for }r\in\{ 1,\ldots ,n\}\,.
\end{equation}
\end{lem}

\begin{proof}
This follows from \eqref{e:Hecke&Langlands} and the usual relation (via Newton polygons) between the absolute values of the roots and coefficients of a polynomial over $\BQbar_p$ (e.g., see  \cite[Ch.~II, Prop.~6.3]{Neu}).
\end{proof}

\subsubsection{V.~Lafforgue's proof of \eqref{e:3any x}} \label{sss:Vincent's proof}
We think of $\pi$ as a subspace of the space of automorphic forms. Since $\pi$ is cuspidal, the automorphic forms from $\pi$ are compactly supported modulo the center of $GL(n,\BA )$.

After twisting $\pi$, we can assume that 
\begin{equation}  \label{e:unit central char}
\eta (\BA^{\times}_F)\subset\BZbar_p^{\times}\,, 
\end{equation}
where $\eta : \BA^{\times}_F\to\BQbar_p^{\times}$ is the central character of $\pi$. 
Let $N\subset\pi$ be the $\BZbar_p$-submodule of those automorphic forms from $\pi$ whose values belong to $\BZbar_p$. 
Then $N\otimes_{\BZbar_p}\BQbar_p=\pi$; this follows from \eqref{e:unit central char} because automorphic forms from $\pi$ are compactly supported modulo the center.
 For each $x\in |X|$ the submodule $N^{GL(n,O_x)}$ is stable under the Hecke operators at $x$.
  For any open subgroup $U\subset GL(n,\BA )$, the $\BZbar_p$-module $N^U$ has finite type. So 
 $t_i^x (\pi )\in\BZbar_p$ for all $i\in\{ 1,\ldots ,n\}$ and $x\in |X|\setminus S_\pi$.
 Moreover, $t_n^x (\pi )\in\BZbar_p^\times$ by \eqref{e:unit central char}. So the function $\Newt^x_\pi$ from Lemma~\ref{l:Hecke&slopes} satisfies the inequality $\Newt^x_\pi (r)\ge r(r-n)/2$,
and by~\eqref{e:Newt} we have $ \sum\limits_{i=1}^{n-r} a_i^x (\pi )=-\Newt^x_\pi (r)\le r(n-r)/2$.
This proves \eqref{e:3any x}. \qed

\subsubsection{Reformulation of Theorem~\ref{t:automorphic}  in terms of  Hecke eigenvalues}   \label{sss:Hecke reformulation}
Let $\pi$ be a  cuspidal automorphic representation of $GL(n,\BA_F)$. For each $x\in |X|\setminus S_\pi$ 
let $\Newt^x_\pi :\{ 0,1,\ldots ,n\}\to\BQ$ be the biggest convex function satisfying \eqref{e:2Newt}.
By Lemma~\ref{l:Hecke&slopes}, one can reformulate Theorem~\ref{t:automorphic} as follows:

(i) there exists a function $\Newt^\eta_\pi :\{ 0,1,\ldots ,n\}\to\BQ$ such that $\Newt^x_\pi=\Newt^\eta_\pi$ for almost all $x\in |X|$;

(ii) $\Newt^x_\pi\ge\Newt^\eta_\pi$ for all $x\in |X|\setminus S_\pi$;

(iii) $\Newt^\eta_\pi (r+1)-2\Newt^\eta_\pi (r)+\Newt^\eta_\pi (r-1)\le 1$ for all $r\in\{ 1,\ldots ,n-1\}$.

\begin{ex}   \label{e:PGL(3)}
Let $\pi$ be a cuspidal automorphic representation of $PGL(3,\BA_F)$. For $x\in |X|\setminus S_\pi$ and $i\in\{ 1,2,3\}$ set 
$c_i^x:=v_x(t_i^x (\pi ))$. As explained in \S\ref{sss:Vincent's proof}, it is clear that for all $x\in |X|\setminus S_\pi$ one has $c_1^x, c_2^x  \ge 0$ and $c_3^x=0$. According to Theorem~\ref{t:automorphic}(i,iii), for \emph{almost} all $x\in |X|\setminus S_\pi$ the point $(c_1^x, c_2^x)$ belongs to $A\cup B$, where
\[
A:=\{ (y_1,y_2)\in\BQ^2\,|\, y_1\ge \frac{1}{3}\,,\, y_2\ge \frac{1}{3}\}, \quad B:=\{ (y_1,y_2)\in\BQ^2\,|\, 0\le y_1/2\le y_2\le 2y_1\}.
\]
\end{ex}

\section{Lemmas on algebraic groups} \label{s:alggroups}
We fix an algebraically closed field $E$ of characteristic 0. All group schemes will be over $E$. All vector spaces and representations are assumed finite-dimensional and over $E$. 

\subsection{The group scheme $\widetilde\BG_m$}   \label{sss:tilde G_m}
Set $\widetilde\BG_m:=\Hom (\BQ,\BG_m)=\underset{n}{\underset{\longleftarrow}\lim}\Hom (n^{-1}\BZ , \BG_m )$. 
For each $n\in\BN$ one has 
\[
\Hom (n^{-1}\BZ , \BG_m )\simeq\Hom(\BZ ,\BG_m)=\BG_m.
\]
So for any algebraic group $H$ the set $\Hom (\widetilde\BG_m ,H)$ canonically identifies with the quotient of the product $\Hom (\BG_m ,H)\times\BN$ by the following equivalence relation: a pair $(f_1,n_1)\in\Hom (\BG_m ,H)\times\BN$ is equivalent to $(f_2,n_2)$ if $f_1^{n_2}=f_2^{n_1}$. Note that $\Hom (\widetilde\BG_m ,\BG_m )=\BQ$, so the weights of $\widetilde\BG_m$ are rational numbers.

\subsection{The small gaps condition}   \label{ss:small gaps}
\subsubsection{}   \label{sss:Vgaps}
We say that a $\widetilde\BG_m$-module $V$ has \emph{small gaps} if the gap between any consecutive weights of $\widetilde\BG_m$ in $V$ is $\le 1$.

\subsubsection{}   \label{sss:mugaps}
Let $G$ be a connected reductive group and $\check\Lambda_G^{+,\BQ}$ the set of its dominant rational coweights. We say that 
$\check\lambda\in\check\Lambda_G^{+,\BQ}$ has  \emph{small gaps} if $(\check\lambda ,\alpha_i)\le 1$ for every simple root $\alpha_i$ of $G$.
This is equivalent to the condition $\check\rho-\check\lambda\in\check\Lambda_G^{+,\BQ}$, where $\check\rho\in\check\Lambda^{+,\BQ}_G$ is one half of the sum of the positive coroots of $G$ (to see this, recall that $(\check\lambda ,\alpha_i)= 1$ for all $i$).

Now assume that $G$ is reductive but not necessarily connected. Let $G^\circ$ be the neutral connected component of $G$.
One has a canonical bijection between $\check\Lambda_{G^\circ}^{+,\BQ}$ and the set of $G^\circ$-conjugacy classes of elements 
$\mu\in\Hom(\widetilde\BG_m ,G)=\Hom(\widetilde\BG_m ,G^\circ )$; the class of $\mu$ in $\check\Lambda_G^{+,\BQ}$ will be denoted by $[\mu ]$. We say that $\mu :\widetilde\BG_m\to G$ has  \emph{small gaps} if $[\mu ]$ has small gaps. Note that if $G=GL(V)$ this is equivalent to \S\ref{sss:Vgaps}.

\subsection{The parabolics $\fp_\mu^\pm$ and $P_\mu^\pm$}  
Somewhat informally, the parabolics defined below are related to the ``big gaps" of $\mu\in\Hom (\widetilde\BG_m ,G)$. 

\subsubsection{The subalgebras $\fp_\mu^\pm\subset\fg$} \label{sss:parabolics setting}
Let $G$ be a reductive group and $G^\circ$ its neutral connected component.
Let $\mu\in\Hom (\widetilde\BG_m ,G)$.

Set $\fg :=\Lie (G)$. Then $\widetilde\BG_m$ acts on $\fg$ via $\mu$. Consider the weight decomposition
\[
\fg=\bigoplus_{r\in\BQ}\fg_r\,  
\]
corresponding to this $\widetilde\BG_m$-action.
Let $\fp_\mu^+$ (resp.  $\fp_\mu^-$) denote the Lie subalgebra of $\fg$ generated by the subspaces $\fg_r$ for $r\ge -1$ (resp.~$r\le 1$).

\begin{lem}  \label{l:fpmu}
(i) $\fp_\mu^+$ and $\fp_\mu^-$ are parabolic subalgebras of $\fg$ opposite to each other.

(ii) The Lie subalgebra $\fp_\mu^+\cap\fp_\mu^-\subset\fg$ is generated by the subspaces $\fg_r$ for $-1\le r\le 1$.

(iii) Let $T\subset G^\circ$ be a maximal torus containing $\mu (\widetilde\BG_m )$. Choose a basis $\alpha_i$, $i\in I$, in the root system of $(G^\circ,T)$ so that $(\alpha_i ,\mu )\ge 0$. Set $I_{\le 1}:=\{i\in I\,|\,(\alpha_i ,\mu )\le 1\}$. Then the Lie algebra $\fp_\mu^+$ is generated by $\Lie (T)$, the root spaces $\fg_{\alpha_i}$ for $i\in I$, and the spaces $\fg_{-\alpha_i}$ for $i\in I_{\le 1}$; the Lie algebra $\fp_\mu^-$ is generated by $\Lie (T)$, the root spaces $\fg_{-\alpha_i}$ for $i\in I$, and the spaces $\fg_{\alpha_i}$ for $i\in I_{\le 1}$; finally, the Lie algebra $\fp_\mu^+\cap\fp_\mu^-$ is generated by $\Lie (T)$ and the spaces $\fg_{\pm\alpha_i}$ for $i\in I_{\le 1}$.
\end{lem}

\begin{proof}
Statement (iii) is clear. Statements (i)-(ii) follow.
\end{proof}

\subsubsection{The subgroups $P_\mu^\pm$ and $M_\mu$}   \label{sss:Pmu}
Let $G$ and $\mu$ be as in \S\ref{sss:parabolics setting}. The group $\pi_0(G)=G/G^\circ$ acts on $\check\Lambda_{G^\circ}^{+,\BQ}$. 
From now on \emph{we will assume that the  class $[\mu]\in\check\Lambda_{G^\circ}^{+,\BQ}$ is $(G/G^\circ)$-invariant.} This means that the $G^\circ$-conjugacy class of $\mu$ is equal to its $G$-conjugacy class or equivalently, that
\begin{equation}  \label{e:assumption}
G=G^\circ\cdot Z_G(\mu ),
\end{equation}
where $Z_G(\mu )$ is the centralizer of $\mu$ in $G$.

Set 
\[
P_\mu^+:=\{g\in G\,|\,\Ad_g (\fp_\mu^+)=\fp_\mu^+\}, \quad P_\mu^-:=\{g\in G\,|\,\Ad_g (\fp_\mu^-)=\fp_\mu^-\} , \quad M_\mu:=P_\mu^+\cap P_\mu^-.
\]
Then $\Lie (P_\mu^\pm )=\fp_\mu^\pm$. The groups $P_\mu^\pm\cap G^\circ$ and $M_\mu\cap G^\circ$ are connected. It is easy to check that 
\begin{equation}  \label{e:centralizer of mu}
Z_G(\mu)\subset M_\mu\, .
\end{equation}
Combining this with \eqref{e:assumption}, we see that the map $G\to\pi_0(G)$ induces isomorphisms
\[
\pi_0(P_\mu^\pm )\iso\pi_0(G),\quad \pi_0(M_\mu )\iso\pi_0(G).
\]

By Lemma~\ref{l:fpmu}(iii), $M_\mu=G$ if and only if $\mu$ has small gaps in the sense of \S\ref{sss:mugaps}.

\begin{ex}   \label{ex:GL}
Let $G=GL (V)$, where $V$ is a vector space. Let $\mu\in\Hom (\widetilde\BG_m ,G)$, then $\widetilde\BG_m$ acts on $V$. It is clear that the $\widetilde\BG_m$-module $V$ has a unique decomposition into a direct sum of $\widetilde\BG_m$-submodules $V_1,\ldots , V_m$ with the following properties:

(a) each of the the $\widetilde\BG_m$-modules $V_i$ has small gaps;

(b) if $r\in\BQ$ is a weight of $\widetilde\BG_m$ on $V_i$ and $r'\in\BQ$ is a weight of $\widetilde\BG_m$ on $V_{i+1}$ then $r-r'>1$.

\noindent In this situation the parabolic $P_\mu^+$ defined in \S\ref{sss:Pmu} is the stabilizer of the flag formed by the subspaces
$V_{\le j}:=\bigoplus\limits_{i\le j}V_i$, and $P_\mu^-$ is the stabilizer of the flag formed by the subspaces
$V_{\ge j}$.
\end{ex}

\begin{rem}   \label{r:functoriality}
Let $G$ and $\mu$ be as in \S\ref{sss:Pmu}. Let $G'$ be a reductive group and $\rho :G\to G'$ a homomorphism. We claim that 
$$\rho (P_\mu^\pm)\subset P_{\rho\circ\mu}^\pm \, .$$
Indeed, since $P_\mu^\pm =(P_\mu^\pm \cap G^\circ )\cdot Z_G(\mu )$ it suffices to check that
$\rho (P_\mu^\pm \cap G^\circ )\subset P_{\rho\circ\mu}^\pm$. This follows from the  inclusion $\rho (\fp_\mu^\pm )\subset \fp_{\rho\circ\mu}^\pm$, which holds by the definition of $\fp_\mu^\pm$.
\end{rem}

\subsubsection{Tannakian approach to $P_\mu^\pm$}   \label{sss:Tannakian}
Let $G$ and $\mu$ be as in \S\ref{sss:Pmu}. Then we have the parabolics $P_\mu^\pm$ defined in \S\ref{sss:Pmu}. On the other hand, for every representation $\rho :G\to GL(V)$ we have the parabolics 
$P_{\rho\circ\mu}^\pm\subset GL(V)$ corresponding to $\rho\circ\mu\in \Hom (\widetilde\BG_m ,GL(V))$ (see Example~\ref{ex:GL} for their explicit description).

\begin{lem}    \label{l:Tannakian}
$P_\mu^\pm=\bigcap\limits_{\rho\in\hat G}\rho^{-1} (P_{\rho\circ\mu}^\pm)$, where $\hat G$ is the set of isomorphism classes of irreducible representations of $G$. 
\end{lem}

\begin{proof}
For any $\rho\in\hat G$ one has $P_\mu^\pm\subset\rho^{-1} (P_{\rho\circ\mu}^\pm)$ by Remark~\ref{r:functoriality}.

Let us construct $\rho_+,\rho_-\in\hat G$ such that 
\begin{equation}   \label{e:desired inclusions}
P_\mu^+\supset\rho_+^{-1} (P_{\rho_+\circ\mu}^+) ,\quad
P_\mu^-\supset\rho_-^{-1} (P_{\rho_+\circ\mu}^-).
\end{equation}
Let $T$, $I$ and $I_{\le 1}$ be as in Lemma~\ref{l:fpmu}(iii). The subset $I_{\le 1}\subset I$ is stable under the action of the group
$G/G^\circ=Z_G(\mu)/Z_{G^\circ}(\mu)$. 
So there exists a $(G/G^\circ)$-invariant dominant weight $\omega$ of $G^\circ$ such that
\[
\{i\in I\,|\, (\omega ,\check\alpha_i)=0\}=I_{\le 1}\, .
\]
Let $\rho^\circ\in \widehat{G^\circ}$ have highest weight $\omega$. Then there exists an irreducible $G$-module $V$ whose restriction to $G^\circ$ is a multiple of $\rho^\circ$. We have a homomorphism $\rho_+:G\to GL(V)$ and the dual homomorphism $\rho_-:G\to GL(V^*)$. We claim that the inclusions \eqref{e:desired inclusions} hold.

Let us check this for $\rho_+$. 
Let $ V_\omega\subset V$ be the highest weight subspace (i.e., the maximal subspace on which $T$ acts by $\omega$), and let $\Stab (V_\omega)\subset G$ be the stabilizer of $V_\omega$ in $G$. Then 
$\rho_+^{-1} (P_{\rho_+\circ\mu}^+) \subset\Stab (V_\omega)$ because $(\omega-\omega' ,\mu)>1$ for any weight $\omega'$ of $T$ in $V_\omega$ such that $\omega'\ne\omega$. On the other hand, 
$\Stab (V_\omega)\cap G^\circ=P_\mu^+\cap G^\circ$, so $\Stab (V_\omega)$ is contained in the normalizer of
$P_\mu^+\cap G^\circ$ in $G$, which equals $P_\mu^+$.
\end{proof}

\begin{cor}   \label{c:small gaps via reps}
Let $G$ and $\mu$ be as in \S\ref{sss:Pmu}.  Suppose that every irreducible $G$-module has small gaps as a $\widetilde\BG_m$-module (with the $\widetilde\BG_m$-action being defined via $\mu:\widetilde\BG_m\to G$). Then $\mu$ has small gaps. \qed
\end{cor}

Let us note that the converse of Corollary~\ref{c:small gaps via reps} is also true (and easy to prove).

\begin{cor}  \label{c:partial functoriality}
Let $\rho: G\to G'$ be a homomorphism of reductive groups. In addition, assume that $G$ is semisimple.
If $\mu\in\Hom (\widetilde\BG_m ,G)$ has small gaps then so does $\rho\circ \mu\in\Hom (\widetilde\BG_m,G')$.
\end{cor}

\begin{proof}
We can assume that $G$ and $G'$ are connected. Applying Corollary~\ref{c:small gaps via reps} to $G'$, we reduce the proof to the case that $G'=GL(V)$.

Since $G$ is semisimple the biggest weight of $\widetilde\BG_m$ in any $G$-module is non-negative and the smallest one is non-positive (because the sum of all weights is zero). This allows us to reduce to the case that $\rho :G\to GL(V)$ is irreducible. This case is clear.
\end{proof}

\subsubsection{Elements of $\Hom (\widetilde\BG_m ,G)$ centralizing each other}   \label{sss:commuting 1-param subgroups}

Let $G$, $\mu$, $\fg$, and $\fg_r$ be as in \S\ref{sss:parabolics setting}. 
Let $G_0$ denote the centralizer of $\mu$ in $G^\circ$; then $\Lie (G_0)$ equals the weight space $\fg_0$. The group $G_0$ is connected and reductive (in fact, it is a Levi of $G^\circ$).

Now let $\nu\in\Hom (\widetilde\BG_m ,G_0)$. Then the product $\mu\cdot\nu$ is a homomorphism $\widetilde\BG_m\to G$. Set $\mu':=\mu\cdot\nu$.

\begin{lem} \label{l:commuting 1-param subgroups}
Suppose that $\nu\in\Hom (\widetilde\BG_m ,G_0)$  has small gaps (see \S\ref{sss:mugaps}) and $\nu (\widetilde\BG_m)\subset [G_0,G_0]$. Then

(i) $\fp_{\mu'}^\pm\supset\fp_\mu^\pm$;

(ii) if $\mu$ has small gaps then so does $\mu'$.
\end{lem}

\begin{proof}
It suffices to prove (i).
We will show that $\fp_{\mu'}^+\supset\fp_\mu^+$ (the inclusion $\fp_{\mu'}^-\supset\fp_\mu^-$ is proved similarly). This is equivalent to proving that $\fp_{\mu'}^+\supset\fg_r$ for all $r\ge -1$.

It is clear that for every $r\ge -1$ one has $\fp_{\mu'}^+\supset\fg_r\cap\fg^{\ge 0}$, where $\fg^{\ge 0}\subset\fg$ is the sum of the non-negative weight spaces of the action $\widetilde\BG_m\overset{\nu}\longrightarrow G_0\to\Aut\fg$. Since $\nu$  has small gaps one has
 $\fp_{\mu'}^+\supset\fg_0$. So for each $r\ge -1$ one has $\fp_{\mu'}^+\supset V_r$, where $V_r\subset\fg_r$ is the $G_0$-submodule generated by $\fg_r\cap\fg^{\ge 0}$. It remains to show that $V_r=\fg_r$. Assume the contrary; then the composition
\[
\widetilde\BG_m\overset{\nu}\longrightarrow G_0\to\Aut(\fg/V_r)\overset{\det}\longrightarrow\BG_m
\]
is strictly negative, which contradicts the assumption $\nu (\widetilde\BG_m)\subset [G_0,G_0]$.
\end{proof}

\subsection{The key lemma} 
The goal of this subsection is to prove Lemma~\ref{l:key}, which will be used in the proof of 
Theorem~\ref{t:2l-adic}.

\begin{lem}    \label{l:H and Mmu}
Let $G$ and $\mu$ be as in \S\ref{sss:Pmu}. Let $H\subset G$ be a subgroup  with the following properties:

(i) $H\supset\mu (\widetilde\BG_m )$;

(ii) every indecomposable $H$-module has small gaps as a $\widetilde\BG_m$-module (with the $\widetilde\BG_m$-action being defined via 
$\mu:\widetilde\BG_m\to H$);

Then $H\subset M_\mu$, $Z_G(H)\subset M_\mu$, and 
\begin{equation}  \label{e:the 2 centers}
Z(M_\mu )\subset Z(Z_G(H)).
\end{equation}
\end{lem}

 Here $Z(M_\mu )$ is the center of $M_\mu$, and $Z(Z_G(H))$ is the center of the centralizer of $H$.

\begin{proof}
By property (ii), for every representation $\rho:G\to GL(V)$ one has $\rho (H)\subset M_{\rho\circ\mu}$. By Lemma~\ref{l:Tannakian}, this implies that $H\subset M_\mu$.

By property (i),  $Z_G(H)\subset Z_G(\mu )$. By \eqref{e:centralizer of mu}, $Z_G(\mu )\subset M_\mu$. So $Z_G(H)\subset M_\mu$.
This implies that  $Z(M_\mu )$ centralizes $Z_G(H)$. On the other hand, $Z(M_\mu )\subset Z_G(H)$ because $H\subset M_\mu$.
So $Z(M_\mu )\subset Z(Z_G(H))$.
\end{proof}

\begin{lem}     \label{l:key}
Let $G, \mu ,H$ be as in Lemma~\ref{l:H and Mmu}. In addition, assume that the inclusion 
$$\Hom (\BG_m ,Z(G))\subset\Hom (\BG_m ,Z (Z_G(H)))$$ is an equality. Then $\mu :\widetilde\BG_m\to G$ has small gaps.
\end{lem}

\begin{proof}
By \eqref{e:the 2 centers}, the inclusion
$\Hom (\BG_m ,Z(G))\subset\Hom (\BG_m ,Z(M_\mu ))$ is an equality. Therefore 
\begin{equation}   \label{e:leading to contradiction}
\Hom (\widetilde\BG_m  ,Z(G))=\Hom (\widetilde\BG_m ,Z(M_\mu )).
\end{equation}

We have to prove that $M_\mu=G$. Assume the contrary.
It is clear that $\mu  (\widetilde\BG_m )\subset M_\mu^\circ$, where $M_\mu^\circ:=M_\mu\cap G^\circ$.
Let $\bar\mu\in\Hom (\widetilde\BG_m ,Z(M_\mu^\circ ))$ denote the image of $\mu\in \Hom (\widetilde\BG_m ,M_\mu^\circ )$ under the composition
\[
\Hom (\widetilde\BG_m ,M_\mu^\circ )\to\Hom (\widetilde\BG_m ,M_\mu^\circ /[M_\mu^\circ ,M_\mu^\circ ])\iso\Hom (\widetilde\BG_m ,Z(M_\mu^\circ )).
\]
To get a contradiction, we will show that $\bar\mu$ belongs to the r.h.s. of \eqref{e:leading to contradiction} but not to the l.h.s.

It is clear that $\bar\mu$ is invariant under all automorphisms of $G$ preserving $\mu$, so $Z_G(\bar\mu)\supset Z_G(\mu)$. 
Since $\bar\mu\in\Hom (\widetilde\BG_m ,Z(M_\mu^\circ ))$ we also have $Z_G(\bar\mu)\supset M_\mu^\circ$. So by \eqref{e:assumption}, we get $Z_G(\bar\mu)\supset M_\mu$. Therefore $\bar\mu\in\Hom (\widetilde\BG_m ,Z(M_\mu ))$.

The coweight $\mu$ is strictly dominant with respect to $P_\mu^+$ (i.e., $(\alpha ,\bar\mu)>0$ for all roots $\alpha$ such that $\fg_\alpha$ is contained in the Lie algebra of the unipotent radical of $P_\mu^+$). So $\bar\mu$ is strictly dominant with respect to~$P_\mu^+$ (indeed, one can get $\bar\mu$ from $\mu$ by averaging with respect to the action of the Weyl group of $M_\mu$). Since
we assumed that $M_\mu\ne G$, we get $\bar\mu (\widetilde\BG_m)\not\subset Z(G)$.
\end{proof}

\section{On Newton coweights for homomorphisms $\pi_1 (X)\to G(\BQbar_\ell)$, where $G$ is reductive}    
\label{s:l-adic for G}
The goal of this section is to generalize Theorem~\ref{t:l-adic} by replacing $GL(n)$ with an arbitrary reductive group.

We fix a universal cover $\tX\to X$ (this is a scheme which usually has infinite type over $\BF_p$). Set $\Pi:=\Aut (\tX /X)$. 

If one also chooses a geometric point $\xi$ of $X$ and a lift of $\xi$ to $\tX $ then $\Pi$ identifies with $\pi_1 (X,\xi )$.

\subsection{Frobenius elements}

Let $|\tX|$ denote the set of closed points of $\tX$. For each $\tilde x\in |\tX|$ we have the geometric Frobenius $\Fr_{\tilde x}\in\Pi$: this is
the unique automorphism of $\tX$ over $X$ whose restriction to $\{\tilde x\}$ equals the composition $\{\tilde x\}\overset{\varphi}\longrightarrow\{\tilde x\}\mono\tX$, where $\varphi :\{\tilde x\}\to\{\tilde x\}$ is the Frobenius morphism with respect to $x$ (this means that for any regular function $f$ on $\{\tilde x\}$ one has $\varphi*(f)=f^{q_x}$, where $q_x$ is the order of the residue field of $x$). 

The map $|\tX|\to\Pi$ defined by $\tilde x\mapsto \Fr_{\tilde x}$ is $\Pi$-equivariant (we assume that $\Pi$ acts on itself by conjugation). So for $x\in X$ one has the geometric Frobenius $\Fr_x\in\Pi$, which is defined only up to conjugacy.

\subsection{A class of homomorphisms $\Pi\to G(\BQbar_\ell)$}
Let $G$ be an algebraic group over $\BQbar_\ell$.

\begin{rem}   \label{r:algebraic elements}
Let $g\in G(\BQbar_\ell )$. Let $g_{\s}\in G (\BQbar_\ell )$ be the semisimple part of the Jordan decomposition of~$g$, and let 
$\langle g_{\s}\rangle$ be the smallest algebraic subgroup of $G$ containing $g_{\s}$. Then the following conditions are equivalent:

(i) the eigenvalues of the image of $g$ in any representation of $G$ are in $\BQbar$;

(ii) $\chi (g_{\s})\in\BQbar$ for all $\chi\in\Hom (\langle g_{\s}\rangle ,\BG_m)$. 
\end{rem}

\subsubsection{``Algebraic" homomorphisms $\Pi\to G(\BQbar_\ell)$}  \label{sss:Algebraic}
Let $\sigma :\Pi\to G(\BQbar_\ell )$ be a continuous homomorphism.\footnote{It is well known that the image of any continuous homomorphism $\Pi\to G( \BQbar_\ell )$ is defined over some \emph{finite} extension of $\BQ_\ell$ (e.g., see \cite[Prop.~3.2.2]{Dr2}).} Similarly to 
\S\ref{sss:algebraicity}, we say that $\sigma$ is \emph{``algebraic"} if for any $x\in |X|$ the element $\sigma (\Fr_x)\in G(\BQbar_\ell)$ (which is well defined up to conjugacy) satisfies the equivalent conditions of Remark~\ref{r:algebraic elements}.

\subsection{The slope homomorphisms}
Let $\BQbar$ be the algebraic closure of $\BQ$ in $\BQbar_\ell$. Fix a valuation $v:\BQbar^\times\to\BQ$ such that $v(p)=1$. Let $\widetilde\BG_m$ be the pro-torus over $\BQbar_\ell$ with group of characters $\BQ$ (see \S\ref{sss:tilde G_m}). 

\subsubsection{The homomorphism $\mu_{g,w}$}   \label{sss:mu_h}
If $g\in G(\BQbar_\ell )$ satisfies the equivalent conditions of Remark~\ref{r:algebraic elements} then each homomorphism 
$w:\BQbar^\times\to\BQ$ defines a homomorphism
\[
\Hom (\langle g_{\s}\rangle ,\BG_m)\to\BQ ,\quad\chi\mapsto w(\chi (g_{\s})),
\]
which is the same as a homomorphism $\mu_{g,w}:\widetilde\BG_m \to\langle g_{\s}\rangle\subset G$.

\subsubsection{The slope homomorphisms}   \label{sss:slope hom}
Let $\sigma :\Pi\to G(\BQbar_\ell )$ be a continuous homomorphism, which is ``algebraic" in the sense of \S\ref{sss:Algebraic}. Let $\tilde x\in |\tX |$, and let $x\in |X|$ be its image. Applying \S\ref{sss:mu_h} to $g=\sigma (\Fr_{\tilde x} )$ and  $w=v/(\deg x )$, one gets a homomorphism 
$\mu_{\tilde x}:\widetilde\BG_m \to G$. We call it the \emph{slope homomorphism} corresponding to $\tilde x$. (This agrees with the terminology of \cite[Thm.~1.8]{RR}.)

\subsection{The Newton coweights}    \label{ss:Newton coweights}
From now on we assume that $G$ is reductive.
Let $G^\circ $ be the neutral connected component of $G$. 
Let $\sigma :\Pi\to G(\BQbar_\ell )$ be a continuous homomorphism, which is ``algebraic" in the sense of \S\ref{sss:Algebraic}.

For each $\tilde x\in |\tX |$ one has the element
$\mu_{\tilde x}\in\Hom (\widetilde\BG_m ,G)=\Hom (\widetilde\BG_m ,G^\circ )$ defined in \S\ref{sss:slope hom}. Recall that the set of conjugacy classes of homomorphisms $\widetilde\BG_m\to G^\circ $ identifies with the set of dominant rational coweights of $G^\circ $, denoted by 
$\check\Lambda_{G^\circ} ^{+,\BQ}$. The class of $\mu_{\tilde x}$ in $\check\Lambda_{G^\circ} ^{+,\BQ}$ is denoted by $a^{\tilde x}(\sigma )$ and called the \emph{Newton coweight} of $\sigma$ at $\tilde x$.

The map $|\tX |\to\check\Lambda_{G^\circ} ^{+,\BQ}$  defined by $\tilde x\mapsto a^{\tilde x}(\sigma )$ is $\Pi$-equivariant, where $\Pi$ acts on 
$\check\Lambda_{G^\circ} ^{+,\BQ}$ via the composition $\Pi\to G\to G/{G^\circ} \to\Aut (\check\Lambda_{G^\circ} ^{+,\BQ})$. In particular, if $G={G^\circ} $ then $a^{\tilde x}(\sigma )$ depends only on the image of $\tilde x$ in~$|X|$, so we can write $a^x(\sigma )$ for $x\in |X|$.
In the case $G=GL(n)$  a homomorphism $\sigma :\Pi\to GL(n,\BQbar_\ell )$ defines a lisse $\BQbar_\ell$-sheaf~$\E$, and $a^x(\sigma )$ is just the collection of the slopes $a_i^x (\E )$ from \S\ref{sss:l-adic slopes}.

\subsection{A generalization of Theorem~\ref{t:l-adic}(i-ii)}
For any reductive group $G$, we equip $\check\Lambda_{G^\circ}^{+,\BQ}$ with the following partial order: $\check\lambda_1\le\check\lambda_2$ if 
$\check\lambda_2-\check\lambda_1$ is a linear combination of simple coroots with non-negative rational coefficients.  Equivalently, this means that $(\omega,\check\lambda_1)\le (\omega,\check\lambda_2)$ for all dominant weights of 
$G^\circ$.

\begin{prop}  \label{p:semicontinuity}
Let $G$ and $\sigma :\Pi\to G(\BQbar_\ell )$ be as in \S\ref{ss:Newton coweights}. In addition, assume that the composition 
$$\Pi\overset{\sigma}\longrightarrow G\to G/{G^\circ} $$
is surjective.

(i) There exists a unique element $a^\eta (\sigma )\in(\check \Lambda_{G^\circ} ^{+,\BQ})^{G/{G^\circ} }$ with the following property: let $U$ denote the set of all $x\in |X|$ such that $a^{\tilde x} (\sigma )= a^\eta (\sigma )$ for all $\tilde x\in |\tX |$ mapping to $x$, then $U$ is non-empty, and for any curve $C\subset X$ the subset $U\cap |C|$ is open in $|C|$.

(ii)  For all  $\tilde x\in |\tX |$ one has $a^{\tilde x}(\sigma )\le a^\eta (\sigma )$.
\end{prop}

\begin{proof}
Let $\Pi_1:=\sigma^{-1}({G^\circ} )\subset\Pi$, and let $\sigma_1:\Pi_1\to {G^\circ} $ be the restricttion of $\sigma$. Let $X_1:=\tX/\Pi_1$ and $\eta_1\in X_1$ the generic point. We claim that it suffices to prove the proposition for $X_1, \Pi_1, \sigma_1$ instead of $X, \Pi , \sigma$. To show this, one only has to check that the rational coweight $a^{\eta_1} (\sigma_1 )$ provided by statement~(i) is $(G/{G^\circ} )$-invariant. Indeed, uniqueness in statement (i) implies 
$(\Pi/\Pi_1)$-invariance of $a^{\eta_1} (\sigma_1 )$, and $\Pi/\Pi_1=G/{G^\circ} $ by the surjectivity assumption.

So it suffices to prove the proposition in the case $G={G^\circ} $.

Uniqueness in (i) follows from Lemma~\ref{l:2irreducibility}. 
To construct $a^\eta (\sigma )$, we need the following observation. 

For each dominant weight $\omega$ of $G$, let $\rho_\omega$ be an irreducible representation of $G$ with highest weight $\omega$, and let $\E_\omega$ be the lisse $\BQbar_\ell$-sheaf on $X$ corresponding to $\rho_\omega\circ\sigma$. Then
\begin{equation}   \label{e:scalar product omega}
(\omega , a^x(\sigma ))=a_1^x(\E_\omega),
\end{equation}
where $a_1^x(\E_\omega)$ is the maximal slope of $\E_\omega$ at $x$.

Choose dominant weights  $\omega_1,\ldots ,\omega_n$ of $G$ so that each dominant weight of $G$ can be written as a linear combination of $\omega_i$'s with non-negative rational coefficients.
Let $U_i$ denote the set of all $x\in |X|$ such that $a_1^x(\E_{\omega_i} )=a_1^\eta (\E_{\omega_i} )$, where $a_1^\eta (\E_{\omega_i} )$ is as in Theorem~\ref{t:l-adic}. By Lemma~\ref{l:2irreducibility}, $\bigcap\limits_i U_i\ne\emptyset$. By \eqref{e:scalar product omega}, the numbers
\[
(\omega_i , a^x(\sigma )), \quad x\in \bigcap\limits_i U_i ,
\]
do not depend on $x$. So for $x\in\bigcap\limits_i U_i$ the coweight $a^x(\sigma )$ does not depend on $x$. Take $a^\eta (\sigma )$ to be this coweight. Property (i) is clear (note that $U=\bigcap\limits_i U_i )$. By \eqref{e:scalar product omega}, we have 
$(\omega_i, a^x(\sigma ))\le (\omega_i, a^\eta (\sigma ))$ for all $x\in |X|$ and all $i$. This is equivalent to (ii).
\end{proof}

\subsection{A generalization of Theorem~\ref{t:l-adic}(iii)}
\subsubsection{The group $Z_G(\sigma )$} \label{sss:centralizer}
Let $G$ and $\sigma :\Pi\to G(\BQbar_\ell )$ be as in Proposition~\ref{p:semicontinuity}. Let $Z_G(\sigma )\subset G$ denote the centralizer of 
$\sigma$ in $G$.

One can think of $Z_G(\sigma )$ as follows. The homomorphism $\sigma$ defines a $G$-local system $\E$ on $X$ (i.e., a $G$-torsor in the Tannakian category of lisse $\BQbar_\ell$-sheaves on $X$, see \S\ref{sss:torsors in Tannakian}). Then $Z_G(\sigma )=\Aut\E$.

\subsubsection{Formulation of the theorem} 
The center of any group $H$ will be denoted by $Z(H)$. It is clear that  $Z(G)\subset Z(Z_{G}(\sigma ))$, so
\begin{equation}   \label{e:theinclusion}
\Hom (\BG_m ,Z(G))\subset \Hom (\BG_m, Z(Z_G(\sigma ))).
\end{equation}

\begin{theorem}  \label{t:2l-adic}
Suppose that in the situation of Proposition~\ref{p:semicontinuity} the inclusion \eqref{e:theinclusion} is an equality. Then $a^\eta (\sigma )$ has small gaps  in the sense of \S\ref{sss:mugaps}.
\end{theorem}

The proof will be given in \S\ref{sss:proof2l-adic}. Similarly to Corollary~\ref{c:any x} and Remark~\ref{r:weak estimate}, one has the following 

\begin{cor} \label{c:2any x}
In the situation of Theorem~\ref{t:2l-adic} suppose that $G$ is semisimple. Then $a^x (\sigma )\le\check\rho$ for all $x\in |X|$. (As usual, 
$\check\rho\in\check\Lambda^{+,\BQ}_{G^\circ}$ is one half of the sum of the positive coroots of $G$.)  
\end{cor}

\begin{proof} 
By Theorem~\ref{t:2l-adic}, $\check\rho -a^\eta (\sigma )$ is dominant. Since $G$ is semisimple, this implies that $\check\rho -a^\eta (\sigma )\ge 0$.
It remains to use Proposition~\ref{p:semicontinuity}(ii).
\end{proof}

\begin{rem}
If $G$ has type $A_n$, Corollary~\ref{c:2any x} is due to V.~Lafforgue (see Remark~\ref{r:weak estimate}). 
If $G$ has type $B_n$ or $C_n$, Corollary~\ref{c:2any x} easily follows from V.Lafforgue's result. This does not seem to be the case for other groups $G$.
\end{rem}

\subsubsection{The case $G=GL(n)$} 
In this case Theorem~\ref{t:2l-adic} is equivalent to Theorem~\ref{t:l-adic}(iii). To deduce it from Theorem~\ref{t:l-adic}(iii), assume that 
$a_i^\eta (\sigma )-a_{i+1}^\eta (\sigma )>1$ for some $i$. Then Theorem~\ref{t:l-adic}(iii) implies that the representation $\sigma$ admits a decomposition $\sigma=\sigma_1\oplus\sigma_2$ such that the slopes of $\sigma_1$ (resp.~$\sigma_2$) at $\eta$ are the numbers $a_j^\eta (\sigma )$ for $j\le i$ (resp. $j>i$). This decomposition is unique, so it is preserved by automorphisms of~$\sigma$. So it yields a homomorphism 
$\BG_m\to Z(\Aut\sigma )=Z(Z_{GL(n)}(\sigma ))$ whose image is not contained in $Z(GL(n))$.

\subsubsection{Proof of Theorem~\ref{t:2l-adic}} \label{sss:proof2l-adic}
 Let $H\subset G$ be the Zariski closure of~$\sigma (\Pi )$. Let $U$ be as in Proposition~\ref{p:semicontinuity}. Fix $\tilde x\in |\tX |\times_{|X|}U$ and set $\mu:=\mu_{\tilde x}$, where $\mu_{\tilde x}:\widetilde\BG_m \to G$ is the slope homomorphism (see \S\ref{sss:slope hom}). Let us check that $\mu$ and $H$ satisfy the conditions of Lemma~\ref{l:key}.
 
 Let $[\mu ]\in\check\Lambda_{G^\circ} ^{+,\BQ}$ be the class of $\mu$. Since $\tilde x\in U$ we have
 \begin{equation}   \label{e:class of mu generic}
 [\mu ]=a^{\tilde x} (\sigma )=a^\eta (\sigma ).
 \end{equation}
By Proposition~\ref{p:semicontinuity}, this  implies that $[\mu ]$ is $(G/G^\circ)$-invariant, as required in \S\ref{sss:Pmu}.

Clearly $H\supset\mu (\widetilde\BG_m )$. We have $Z_G(\sigma )=Z_G(H)$, so \eqref{e:theinclusion} means that
\[
\Hom (\BG_m ,Z(G))=\Hom (\BG_m, Z(Z_G(H)).
\]
Finally, by Theorem~\ref{t:l-adic}(iii), every indecomposable $H$-module has small gaps as a $\widetilde\BG_m$-module 
(if the $\widetilde\BG_m$-action is defined via $\mu$).

Applying Lemma~\ref{l:key}, we see that $\mu$ has small gaps. By \eqref{e:class of mu generic}, this means that $a^\eta (\sigma )$ has small gaps.
 \qed

\subsection{A corollary related to elliptic Arthur parameters} \label{ss:Arthur}

\subsubsection{A class of homomorphisms $\Pi\times SL(2,\BQbar_\ell)\to G(\BQbar_\ell )$}   \label{sss:A class of}
We will be considering homomorphisms 
$\psi :\Pi\times SL(2,\BQbar_\ell)\to G(\BQbar_\ell )$ satisfying the following conditions:

(i) the restriction of $\psi$ to $\Pi$ is continuous and ``algebraic" in the sense of \S\ref{sss:Algebraic};

(ii) the restriction of $\psi$ to $SL(2,\BQbar_\ell)$ is a homomorphism of algebraic groups;

(iii) the map $\Pi\to G/G^\circ$ induced by $\psi$ is surjective;

(iv) the centralizer of $\psi$ in $G$ is finite modulo the center of $G$.

\subsubsection{Relation to Arthur parameters}   \label{sss:Arthur parameters}
Suppose that $\dim X=1$ and that $G$ is equipped with a splitting $G/{G^\circ}\to G$.  According to 
\cite[\S 12.2.2]{VLa2}, a homomorphism $\psi$ satisfying (i)-(iii) is called an \emph{Arthur parameter} if the restriction of $\psi$ to $\Pi$ becomes pure of weight $0$ when composed with any representation of $G$.  An Arthur parameter satisfying (iv) is said to be \emph{elliptic} (or \emph{discrete}, see the Appendix of \cite{BC}). 

\subsubsection{The homomorphism $\phi_\psi :\Pi\to G(\BQbar_\ell )$}  \label{sss:Langlands}
Let $\tau :\Pi\to\BQ_{\ell}^\times$ be the cyclotomic character. Fix a square root $\tau^{1/2}:\Pi\to\BQbar_{\ell}^\times$ whose restriction to the geometric part of $\Pi$ is trivial. Given $\psi$ as in \S\ref{sss:A class of}, define a homomorphism $\phi_\psi :\Pi\to G(\BQbar_\ell )$ by
\[
\phi_{\psi}(\gamma)=\psi (\gamma,\xi (\tau^{1/2}(\gamma ))),
\]
where $\xi:\BG_m\to SL(2)$ is the homomorphism
\begin{equation}   \label{e:xi}
\xi (t):=\begin{pmatrix} t  & 0 \\
 0 & t^{-1}
 \end{pmatrix}.
\end{equation}

Let us note that if $\psi$ is an Arthur parameter (see \S\ref{sss:Arthur parameters}) then $\phi_{\psi}$ is called the \emph{Langlands parameter associated to $\psi$.}
 
\begin{cor}   \label{c:Arthur}
Let $\psi :\Pi\times SL(2,\BQbar_\ell)\to G(\BQbar_\ell )$ be a homomorphism satisfying conditions (i)-(iv) of \S\ref{sss:A class of}. Then 
$a^\eta (\phi_\psi )$ has small gaps in the sense of \S\ref{sss:mugaps}.
\end{cor}

\begin{proof}
Let $\sigma :\Pi\to G(\BQbar_\ell )$ and $f:SL(2)\to G$ be the restrictions of $\psi$. Then 
$\im f\subset Z_G(\sigma )$, so
\[
Z(Z_G(\sigma ))=Z_G(\sigma )\cap Z_G(Z_G(\sigma ))\subset Z_G(\sigma )\cap Z_G(f)=Z_G(\psi).
\]
Therefore condition (iv) of \S\ref{sss:A class of} implies that $\Hom (\BG_m, Z(Z_G(\sigma )))=\Hom (\BG_m ,Z(G))$. 
So by Theorem~\ref{t:2l-adic}, $a^\eta (\sigma )$ has small gaps. 

To deduce from this that $a^\eta (\phi_\psi )$ has small gaps, we will apply Lemma~\ref{l:commuting 1-param subgroups}. 
Let $\mu:\widetilde\BG_m \to G$ be as in~\S\ref{sss:proof2l-adic}; we already know that $\mu$ has small gaps. Let $G_0$ be the centralizer of 
$\mu$ in $G^\circ$. The homomorphism $\psi$ maps $SL(2)$ to $G_0$. Let $\xi$ be as in \eqref{e:xi}, and let $\nu:\widetilde\BG_m \to G_0$ be the composition of $\xi^{1/2}:\widetilde\BG_m \to SL(2)$ and $\psi |_{SL(2)}:SL(2)\to G_0$. By the theory of $\fs\fl (2)$-triples, $\nu$ has small gaps (see \cite[\S VIII.11, Prop.~5]{Bo}); this also follows from  Corollary~\ref{c:partial functoriality}.
Since $SL(2)$ is semisimple, we have  $\nu (\widetilde\BG_m)\subset [G_0, G_0]$. It remains to apply 
Lemma~\ref{l:commuting 1-param subgroups}(ii).
\end{proof}

\section{An analog of Theorem~\ref{t:main} for arbitrary reductive groups}  \label{s:FIsoc_G}

\subsection{Generalities on Tannakian categories}
References: \cite{DM,De90,De3}.

\subsubsection{$G$-torsors in Tannakian categories}   \label{sss:torsors in Tannakian}
If $\cT$ is a Tannakian category over a field $E$ and $G$ is an algebraic group over $E$, then a \emph{$G$-torsor in $\cT$} is an exact tensor functor $\Rep (G)\to\cT$, where $\Rep (G)$ is the tensor category of finite-dimensional representions of $G$. If $\cT$ is the category of vector spaces this is the usual notion of $G$-torsor. A $GL(n)$-torsor in any Tannakian category $\cT$ is the same as an  $n$-dimensional object of  $\cT$.

For any $G$, all $G$-torsors in $\cT$ form a groupoid enriched over the category of $E$-schemes. So for any group scheme $H$ over $E$ there is a notion of $H$-action on a $G$-torsor.

\subsubsection{Example} \label{sss:torsor with action}
A $G$-torsor in $\Rep (H)$ is the same as a usual $G$-torsor equipped with $H$-action. So if $E$ is algebraically closed then $G$-torsors in $\Rep (H)$ are classified by conjugacy classes of homomorphisms $H\to G$.

Let us apply this for $H=\widetilde\BG_m$, where $\widetilde\BG_m$ is 
as in \S\ref{sss:tilde G_m}. If $G$ is a connected reductive group
and $E$ is algebraically closed then the set of conjugacy classes of homomorphisms $\widetilde\BG_m\to G$ identifies with the set of dominant rational coweights of $G$, denoted by $\check\Lambda_G^{+,\BQ}$, so $G$-torsors in $\Rep (\widetilde\BG_m)$ are classified by 
$\check\Lambda_G^{+,\BQ}$.

More generally, let $\cT$ be a Tannakian category over any field $E$ and $\E$ be a $G$-torsor in $\cT$ equipped with a $\widetilde\BG_m$-action. Then one defines the class of $\E$ in $\check\Lambda_G^{+,\BQ}$ to be the class of $F(\E )$, where $F$ is any fiber functor on $\cT$ over an algebraically closed extension of $E$.

\subsection{$G$-torsors in $\FIsoc (X)\otimes_{\BQ_p}\BQbar_p$} Clearly $\FIsoc (X)$ and $\FIsoc (X)\otimes_{\BQ_p}\BQbar_p$ are Tannakian categories over $\BQ_p$ and $\BQbar_p$, respectively.

\subsubsection{Grading by slopes}   \label{sss:pointwise grading}
For any $x\in X$, let $x_{\perf}$ denote the spectrum of the perfection of the residue field of $x$. The Tannakian category $\FIsoc (x_{\perf})$ has a canonical $\BQ$-grading by slopes. The pullback of $M\in\FIsoc (X)\otimes_{\BQ_p}\BQbar_p$ to $x_{\perf}$ will be denoted by $M_x$.

\subsubsection{Definition of $\FIsoc_G (X)$}
Let $G$ be an algebraic group over $\BQbar_p$. The groupoid of $G$-torsors in the Tannakian category $\FIsoc (X)\otimes_{\BQ_p}\BQbar_p$ (see \S\ref{sss:torsors in Tannakian}) will be denoted by $\FIsoc_G (X)$. In \cite{RR} objects of $\FIsoc_G (X)$ are called \emph{$F$-isocrystals with $G$-structure.}

\subsubsection{Examples} \label{sss:examples of torsors}
(i) A $GL(n)$-torsor in $\FIsoc (X)\otimes_{\BQ_p}\BQbar_p$ is the same as an  $n$-dimensional object of  $\FIsoc (X)\otimes_{\BQ_p}\BQbar_p$.

(ii) If $G$ is finite then by Proposition~\ref{p:Crew}, a $G$-torsor in $\FIsoc (X)\otimes_{\BQ_p}\BQbar_p$ is the same as a $G$-torsor on $X_{\et}$.

\subsubsection{Newton coweights if $G$ is connected}    \label{sss:crys-Newton}
Let $G$ be a connected reductive group over $\BQbar_p$ and $\E\in\FIsoc_G (X)$. Let $x\in X$. Since $\FIsoc (x_{\perf})$ is $\BQ$-graded, the $G$-torsor $\E_x$ is equipped with an action of $\widetilde\BG_m:=\Hom (\BQ ,\BG_m)$. So by \S\ref{sss:torsor with action}, it has a class 
$a^x (\E )\in\check\Lambda_G^{+,\BQ}$. Following \cite{RR,Ko2}, we call it the \emph{Newton coweight} of $\E_x$ (or the Newton coweight of $\E$ at $x$). If $G=GL(n)$ then $\E $ is just an $n$-dimensional object of $\FIsoc (X)\otimes_{\BQ_p}\BQbar_p$ and $a^x (\E )$ is the collection of its slopes $a_i^x(\E )$, $1\le i\le n$.

\subsubsection{Newton coweights without assuming connectedness of $G$}    \label{sss:disconnected-crys-Newton}
We fix a universal cover $\tX\to X$ and set $\Pi:=\Aut (\tX /X)$. 

Now let $G$ be a not necessarily connected reductive group over $\BQbar_p$; let $G^\circ$ be its neutral connected component. An object 
$\E\in\FIsoc_G (X)$ defines an object $\bar\E\in\FIsoc_{\pi_0(G)} (X)$, which is just a $\pi_0(G)$-torsor on $X_{\et}$, see \S\ref{sss:examples of torsors}(ii). Fix a trivialization of the pullback of $\bar\E$ to $\tX$. Then the torsor $\bar\E$ is described by a homomorphism 
\begin{equation}   \label{e:Pi to pi_0}
\Pi\to \pi_0 (G).
\end{equation}
Let $U$ be its kernel. The pullback of $\bar\E$ to 
$\tX/U$ is trivialized, so the pullback of $\E$ to $\tX/U$ comes from a $G^\circ$-torsor. Applying \S\ref{sss:crys-Newton} to this $G^\circ$-torsor on $\tX/U$, we get a map $\tX/U\to \check\Lambda_{G^\circ}^{+,\BQ}$. The corresponding map $\tX\to \check\Lambda_{G^\circ}^{+,\BQ}$ is denoted by
\begin{equation}   \label{e:disconnected-crys-Newton}
\tilde x\mapsto a^{\tilde x} (\E ), \quad \tilde x\in\tX .
\end{equation}
The composition of \eqref{e:Pi to pi_0} and the homomorphism $\pi_0 (G)\to\Aut\check\Lambda_{G^\circ}$ defines an action of $\Pi$ on $\Lambda_{G^\circ}^{+,\BQ}$. The map \eqref{e:disconnected-crys-Newton} is $\Pi$-equivariant.

Let $\eta,\tilde\eta$ denote the generic points of $X,\tX$. We write $a^{\eta} (\E )$ instead of 
$a^{\tilde \eta} (\E )$. By $\Pi$-equivariance of \eqref{e:disconnected-crys-Newton}, we have
\begin{equation}   \label{e:Pi-invariance}
a^{\eta} (\E )\in (\Lambda_{G^\circ}^{+,\BQ})^\Pi.
\end{equation}

\subsubsection{Remark}   \label{sss:Newton in terms of Crew}
Here is a reformulation of the definition of $a^{\tilde x} (\E )$ in terms of \S\ref{ss:crys fund group} of Appendix~\ref{s:Crew}. An object $\E\in\FIsoc_G (X)$ equipped with a trivialization of the pullback of $\bar\E$ to $\tX$ is the same as a homomorphism $\sigma :\pi_1^{\FIsoc}(X)\to G$ up to $G^\circ$-conjugation. We also have a homomorphism $\nu_{\tilde x}:\widetilde\BG_m\to\pi_1^{\FIsoc}(X)$ defined up to 
conjugation by $(\pi_1^{\FIsoc}(X))^\circ$, see the end of \S\ref{sss:Some homs}. Set 
$\mu_{\tilde x}:=\sigma\circ\nu_{\tilde x}\in\Hom (\widetilde\BG_m ,G^\circ )$.
Then  $a^{\tilde x} (\E )\in\check\Lambda_{G^\circ}^{+,\BQ}$ is the conjugacy class of $\mu_{\tilde x}$.

\subsubsection{Semicontinuity}
The set $\check\Lambda_G^{+,\BQ}$ is equipped with the following partial order: $\check\lambda_1\le\check\lambda_2$ if $\check\lambda_2-\check\lambda_1$ is a linear combination of simple coroots with non-negative rational coefficients. 
The map \eqref{e:disconnected-crys-Newton} is \emph{lower-semicontinuous}. It suffices to check this if $G$ is connected; in this case this is Theorem~3.6 of \cite{RR} (as explained in \cite{RR}, it immediately follows from semicontinuity of usual Newton polygons, see \cite[Cor.~2.3.2]{K}). 
In particular, lower semicontinuity implies that
\begin{equation}   \label{e:special&generic}
a^{\tilde x} (\E )\le a^\eta (\E) \quad \mbox { for all }\tilde x\in\tX .
\end{equation}

Let us note that the map \eqref{e:disconnected-crys-Newton} takes finitely many values. It suffices to check this for 
$G=GL(n)$; in this case it follows from \eqref{e:special&generic} and boundedness of the denominators of the slopes.

\begin{theorem}   \label{t:G-version of main}
Let $G$ be a reductive group and $\E\in\FIsoc_G (X)$. Suppose that

(i) the map \eqref{e:Pi to pi_0} is surjective;

(ii) the inclusion
\begin{equation}   
\Hom (\BG_m ,Z(G))\subset \Hom (\BG_m, Z(\Aut\E))
\end{equation}
is an equality (here $Z(G)$ and $Z(\Aut\E)$ are the centers of $G$ and $\Aut\E$, respectively).

Then $a^\eta (\E )$ has small gaps in the sense of \S\ref{sss:mugaps}.
\end{theorem}

\begin{proof}
We will mimic the proof of Theorem~\ref{t:2l-adic} given in \S\ref{sss:proof2l-adic}. 

Our $\E$ corresponds to a homomorphism  $\sigma :\pi_1^{\FIsoc}(X)\to G$, where $\pi_1^{\FIsoc}(X)$ is the group defined in \S\ref{ss:crys fund group} of Appendix~\ref{s:Crew}. Let $H:=\im (\sigma)\subset G$. 

In \S\ref{sss:Newton in terms of Crew} we defined $\mu_{\tilde x}\in\Hom (\widetilde\BG_m ,G^\circ )$ for each $\tilde x\in\tX$. Let 
$\mu:=\mu_{\tilde\eta}$, where $\tilde\eta\in\tX$ is the generic point. Then  $a^{\eta} (\E )$ is equal to the conjugacy class
$[\mu]\in\check\Lambda_{G^\circ}^{+,\BQ}$.

It remains to check that $\mu$ and $H$ satisfy the conditions of Lemma~\ref{l:key}. 

Combining \eqref{e:Pi-invariance} and assumption (i) of the theorem, we see that $[\mu]$ is $\pi_0(G)$-invariant, as required in \S\ref{sss:Pmu}.
Clearly $H\supset\mu (\widetilde\BG_m )$. Assumption (ii) of the theorem means that
\[
\Hom (\BG_m ,Z(G))=\Hom (\BG_m, Z(Z_G(H))).
\]
Finally, by Theorem~\ref{t:main} combined with Corollary~\ref{c:remain valid}, every indecomposable $H$-module has small gaps as a $\widetilde\BG_m$-module (if the $\widetilde\BG_m$-action is defined via $\mu$).
\end{proof}

The interested reader can formulate and prove an analog of Corollary~\ref{c:Arthur} for $\FIsoc_G (X)$.

\section{Newton weights of automorphic representations of reductive groups}   \label{s:2automorphic}

\subsection{Notation}  \label{sss:2automorphic notation}
We keep the notation of \S\ref{sss:automorphic notation}-\ref{sss:sqrt}. In particular, 
$v:\BQbar_p^\times\to\BQ$ is the $p$-adic valuation normalized so that $v(p)=1$, and for each $x\in |X|$ we define $v_x:\BQbar_p^\times\to\BQ$ by $v_x:=v/\deg x$.

Just as in \S\ref{s:l-adic for G}, we fix a universal cover $\tX\to X$ and set $\Pi:=\Aut (\tX /X)$. For each $\tilde x\in |\tX|$, one has the geometric Frobenius $\Fr_{\tilde x}\in\Pi$.

\subsection{The groups $G$, $\check G$, and $\null^LG$}
\subsubsection{The group scheme $G$}
Let $G$ be a smooth group scheme over $X$ with connected reductive fibers. If $A$ is a ring equipped with a morphism $\Spec A\to X$ we write $G(A):=\Mor_X(\Spec A,G)$ (this can be applied for rings $F$, $\BA_F$, or $O_x$, $x\in |X|$).

\subsubsection{Weights and coweights}   \label{sss:weights coweights}
Let $\Lambda_G$ (resp.~$W$) be the group of weights (resp.~ Weyl group) of $G\times_X\tX$.
The group $\Pi$ acts on $\Lambda_G$ and $W$. As usual, let $\check\Lambda_G$ be the group dual to $\Lambda_G$. 

Let $\Lambda_G^{\BQ}:=\Lambda_G\otimes\BQ$, and let $\Lambda_G ^{+,\BQ}\subset\Lambda_G^\BQ$ be the dominant cone; one has the usual bijection $\Lambda_G ^{+,\BQ}\iso (\Lambda_G\otimes\BQ)/W$.
The set $\Lambda_G^{+,\BQ}$ is equipped with the following partial order: $\lambda_1\le\lambda_2$ if 
$\lambda_2-\lambda_1$ is a linear combination of simple roots with non-negative rational coefficients. 

We say that 
$\lambda\in\Lambda_G^{+,\BQ}$ has  \emph{small gaps} if $(\lambda ,\check\alpha_i)\le 1$ for every simple coroot $\check\alpha_i$.

\subsubsection{The groups $\check G$ and $\null^LG$}
Let $\check G$ be the split reductive group over $\BQ$ canonically associated to the root datum dual to that of~$G\times_X\tX$ (in particular, $\Lambda_{\check G}=\check\Lambda_G$). 

Let $\Pi':=\Ker (\Pi\to\Aut\Lambda_G)$, then $\Pi/\Pi'$ acts on $\check G$. Set 
$\null^LG:=(\Pi/\Pi')\ltimes \check G$; this is the \emph{Langlands dual} of $G$.

\subsubsection{The variety $[\null^LG ]$}
Define $[\null^LG ]$ to be the GIT quotient of $\null^LG $ by the conjugation action of $\check G$ (i.e., $[\null^LG]$ is the spectrum of the algebra of those regular functions on $\null^LG$ that are invariant under $\check G$-conjugation). This is a variety equipped with an action of $\Pi/\Pi'$ and a $(\Pi/\Pi')$-equivariant map $[\null^LG ]\to \Pi/\Pi'$. Moreover, the action of each element $\gamma\in\Pi/\Pi'$ on its preimage in $[\null^LG ]$ is trivial.

\subsection{Satake parameters}  \label{ss:Satake}
Let $x\in |X|$, and let $E$ be an algebraically closed field of characteristic $0$. 
Unramified irreducible representations of $G(F_x)$ over $E$ are classified by elements of a certain set
$\Sat_x (E)$, whose elements are called \emph{Satake parameters} at $x$.

$\Sat_x (E)$ is defined as follows. Let $[\tX ]_x$ denote the preimage of $x$ in $[\tX ]$. We have a canonical map 
\[
[\tX ]_x\to\Pi/\Pi', \quad\tilde x\mapsto \overline{\Fr}_{\tilde x},
\]
where $\overline{\Fr}_{\tilde x}$ is the image of $\Fr_{\tilde x}$ in $\Pi/\Pi'$. An element of $\Sat_x (E)$ is a  lift of this map to a $\Pi$-equivariant map 
$[\tX ]_x\to [\null^LG ] (E)$. (To specify such a map $[\tX ]_x\to [\null^LG ](E)$, it suffices to specify the image of a single point $\tilde x\in[\tX ]_x$, and this image can be any element of $[\null^LG ](E)$ over $\overline{\Fr}_{\tilde x}$.)

It is well known \cite{Lan,Mo,Sp} that for every $\tilde x\in |\tX |_x$ one has a canonical bijection 
\[
\Sat_x (E)\iso \Hom (\check\Lambda_G^{\Fr_{\tilde x}}, E^\times )/W^{\Fr_{\tilde x}}, 
\]
where $\check\Lambda_G^{\Fr_{\tilde x}}$ and $W^{\Fr_{\tilde x}}$ are the invariants of $\Fr_{\tilde x}$ acting on $\check\Lambda_G$ and $W$. 

\subsection{Newton weights}   \label{sss:automorphic Newton}
Let $\pi$ be an irreducible admissible representation of $G(\BA_F)$ over $\BQbar_p$, which is unramified at each $x\in |X|$  (i.e., the subspace of $G(O_x)$-invariant vectors is not zero). We will associate to $\pi$ a $\Pi$-equivariant map 
\begin{equation} \label{e:Newton adelic}
|\tX |\to\Lambda_G ^{+,\BQ}, \quad \tilde x\mapsto a^{\tilde x}(\pi).
\end{equation}
The element $a^{\tilde x}(\pi )\in\Lambda_G ^{+,\BQ}$ will be called the \emph{Newton weight} of $\pi$ at $\tilde x$. 

To define the map \eqref{e:Newton adelic}, it suffices to define its restriction to $|\tX |_x$ for each $x\in |X|$. We define it to be the composition 
\[
[\tX ]_x\to [\null^LG ](\BQbar_p)\to \Lambda_G ^{+,\BQ}
\]
where the first map is the Satake parameter of $\pi_x$ (see \S\ref{ss:Satake}), and the second map is defined as follows.\footnote{The definition given below is similar to \S\ref{sss:mu_h}.} An element of $[\null^LG ](\BQbar_p)$ is a 
$\check G(\BQbar_p)$-conjugacy class of semisimple elements
$g\in\null^LG (\BQbar_p)$. For such $g$, let $\langle g\rangle$ be the smallest algebraic subgroup of 
$\null^L G\otimes\BQbar_p$ containing 
$g$; then the homomorphism
\[
\Hom (\langle g\rangle ,\BG_m)\to\BQ ,\quad\chi\mapsto v_x(\chi (g))
\]
defines a homomorphism $\widetilde\BG_m \to\langle g\rangle\subset\null^L  G\otimes\BQbar_p$, where  $\widetilde\BG_m$ is the pro-torus 
over $\BQbar_p$ with group of characters $\BQ$. Thus an element of 
$[\null^LG ](\BQbar_p)$ defines a $\check G(\BQbar_p)$-conjugacy class of homomorphisms 
$\widetilde\BG_m\to\null^LG\otimes_{\BQ_p}\BQbar_p$, which is the same as an element of 
$\check\Lambda_{\check G}^{+,\BQ}=\Lambda_G ^{+,\BQ}$.

\begin{rem}  \label{r:automorphic Newton}
The above definition of $a^{\tilde x}(\pi)$ is equivalent to the following one. The Satake parameter of $\pi_x$ is an element of 
$\Hom (\check\Lambda_G^{\Fr_{\tilde x}}, \BQbar_p^\times )/W^{\Fr_{\tilde x}}$. Now consider the composition
\begin{equation}   \label{e:thecomp}
\Hom (\check\Lambda_G^{\Fr_{\tilde x}}, \BQbar_p^\times )\to \Hom (\check\Lambda_G^{\Fr_{\tilde x}}, \BQ )=
(\Lambda_G^{\BQ})_{\Fr_{\tilde x}}\iso(\Lambda_G^{\BQ})^{\Fr_{\tilde x}}\mono \Lambda_G^{\BQ},
\end{equation}
where the first map is induced by $v_x:\BQbar_p^\times\to\BQ$ and the isomorphism 
$(\Lambda_G^{\BQ})_{\Fr_{\tilde x}}\iso(\Lambda_G^{\BQ})^{\Fr_{\tilde x}}$ is the averaging operator, i.e., the inverse of the tautological map
$(\Lambda_G^{\BQ})^{\Fr_{\tilde x}}\iso(\Lambda_G^{\BQ})_{\Fr_{\tilde x}}\,$. The composition \eqref{e:thecomp} is $W^{\Fr_{\tilde x}}$-equivariant, so it induces a map $ \Hom (\check\Lambda_G^{\Fr_{\tilde x}}, \BQbar_p^\times )/W^{\Fr_{\tilde x}}\to  \Lambda_G^{\BQ}/W=\Lambda_G ^{+,\BQ}$. Applying this map to the Satake parameter of $\pi_x$, one gets an element of $\Lambda_G ^{+,\BQ}$. This is $a^{\tilde x}(\pi)$. 
\end{rem}

Let us note that the definition of the Satake parameter depends on the choice of $p^{1/2}$ (see \S\ref{sss:sqrt}), but $a^{\tilde x}(\pi )$ does not depend on this choice.

\subsection{The case of a torus}  \label{ss:torus}
Let $T$ be a torus over $X$. Let $\eta :T(\BA_F)\to\BQbar_p^\times$ be a homomorphism with open kernel, which is unramified at every point of $|X|$. 

\begin{lem}
Let $X':=\tX/\Pi'$, where $\Pi':=\Ker (\Pi\to\Aut\Lambda_T)$. Let $F'$ be the field of rational functions on $X'$. Let $\tilde x\in |\tX |$ and $x'\in X'$ the image of~$\tilde x$. Let $\check\lambda\in\check\Lambda_T$. Then
\begin{equation}   \label{e:Newton toric}
(a^{\tilde x}(\eta ),\check\lambda)=f_{\check\lambda}(\varpi_{x'})/\deg x',
\end{equation}
where $f_{\check\lambda}:\BA_{F'}^\times\to\BQ$ is the composition
\[
\BA_{F'}^\times\overset{\check\lambda}\longrightarrow T(\BA_{F'})\overset{\Norm}\longrightarrow T(\BA_F)\overset{\eta}\longrightarrow
\BQbar_p^\times\overset{v}\longrightarrow\BQ
\]
and $\varpi_{x'}$ is a uniformizer in $(F'_{x'})^\times\subset\BA_{F'}^\times$. 
\end{lem}

\begin{proof}
Follows from Remark~\ref{r:automorphic Newton}.
\end{proof}

\begin{cor} \label{c:a(eta)}
If $\eta$ is trivial on $T(F)$ then $a^{\tilde x}(\eta )$ does not depend on $\tilde x\in |\tX |$.
\end{cor}

\begin{proof}
Let us show that the r.h.s. of \eqref{e:Newton toric} does not depend on $x'\in |X'|$. Indeed,
the homomorphism $f_{\check\lambda}:\BA_{F'}^\times\to\BQ$ is trivial on $(F')^\times$, so it is proportional to 
$\deg :\BA_{F'}^\times\to\BZ$.
\end{proof}

In the situation of Corollary~\ref{c:a(eta)} we will write $a(\eta )$ instead of $a^{\tilde x}(\eta )$.

\subsection{V.~Lafforgue's estimate} \label{sss:VLafforgue p-adic G}
Let $G$ be a split reductive group. Let $\pi$ be a cuspidal automorphic irreducible representation of $G (\BA_F)$ over 
$\BQbar_p$, which is unramified at each point of $|X|$. Let $Z^\circ$ be the neutral connected component of the center of $G$, and let 
$\eta : Z^\circ (\BA_F)\to\BQbar_p^\times$ be the central character of $\pi$.
According to \cite[Prop.~2.1]{VLa}, 
\begin{equation}
a^{\tilde x}(\pi )\le a(\eta)+\rho \quad \mbox{ for all } \tilde x\in |\tX |,
\end{equation}
where $\rho$ is one half of the sum of the positive roots of $G$ and $a(\eta)\in\Lambda_{Z^\circ}^\BQ\subset\Lambda_{G}^\BQ$ is defined at the end of~\S\ref{ss:torus}.

Probably the method of \cite{VLa} should work even without assuming $G$ to be split (this is discussed in  \S 5~of~\cite{VLa}).

\subsection{The result}  \label{ss:2automorphic result}
As before, let $\pi$ be an irreducible representation of $G(\BA_F)$ over~$\BQbar_p$, which is unramified at all points 
$x\in |X|$.

\begin{theorem}   \label{t:2automorphic}
Suppose that $\pi$ is cuspidal automorphic. Then

(i) there exists an element $a^\eta (\pi)\in (\Lambda_G ^{+,\BQ})^\Pi$ such that the set $\{\tilde x\in |\tX | \;|\;a^{\tilde x} (\pi )\ne a^\eta (\pi)\}$ has finite image in $|X|$;

(ii) for all  $\tilde x\in |X| $ one has $a^{\tilde x} (\pi )\le a^\eta (\pi)$;

(iii) Conjecture 12.7 of \cite{VLa2} would imply that $a^\eta (\pi)$ has small gaps in the sense of \S\ref{sss:weights coweights}.
\end{theorem}

Conjecture 12.7 of \cite{VLa2} is a variant of Arthur's conjecture \cite{Ar}. 

\begin{proof}
Let $G^{ab}:=[G,G]$. After twisting $\pi$ by a quasi-character of $G^{\ab}(\BA)/G^{\ab}(F)$, we can assume that $\pi$ is defined over the subfield $\BQbar\subset\BQbar_p$. Choose a prime $\ell\ne p$ and an algebraic closure $\BQbar_\ell$ of $\BQ_\ell$ equipped with a homomorphism $\BQbar\mono\BQbar_\ell$. Note that $\BQbar$ is equipped with a distinguished $p$-adic valuation (namely, the restriction of $v:\BQbar_p^\times\to\BQ$).

By Theorem 12.3 of \cite{VLa2}, the canonical epimorphism $\Pi\to\Pi/\Pi'=\null^LG/\check G$ can be lifted to
a continuous homomorphism $\sigma :\Pi\to\null^LG(\BQbar_{\ell})$ which is compatible with $\pi$ in the sense of Langlands; this means that for each $x\in |X|$ the element of $\Sat_x (\BQbar_{\ell})$ corresponding to the map
\[
\tilde x\mapsto\sigma (\Fr_{\tilde x}), \quad \tilde x\in|\tX |_x
\]
is equal to the Satake parameter of $\pi_x$. Applying Proposition~\ref{p:semicontinuity} to $\sigma$, one gets statements (i)-(ii). 

Conjecture 12.7 of \cite{VLa2} would imply that $\sigma$ can be chosen to be the Langlands parameter corresponding to an elliptic Arthur parameter (see \S\ref{sss:Arthur parameters}-\ref{sss:Langlands}). Applying Corollary~\ref{c:Arthur} to such $\sigma$, one gets statement (iii).
\end{proof}

\appendix

\section{Slopes for some hypergeometric local systems} \label{s:Dwork}
The main goal of this Appendix is to prove Proposition~\ref{p:counterexample}, which 
provides  the counterexamples promised in \S\ref{ss:counterex} and Remark~\ref{r:not guaranteed}. In \S\ref{ss:Magma} we explain how to produce more examples using a computer.

We will work over $\BF_p$; in particular, $\BG_m$ denotes the multiplicative group over $\BF_p$.

\subsection{A class of hypergeometric local systems}
\subsubsection{The fields $E$ and $E_\lambda$}
Let $E\subset\BQ_p$ be the subfield generated by the roots of unity of order $p-1$. We fix a non-Archimedean place $\lambda$ of $E$ not dividing $p$. Let $E_\lambda$ denote the completion.

Let $\tau:\BF_p^\times\to E^\times$ denote the Teichm\"uller character. We have a canonical isomorphism
\[
\BZ/(p-1)\BZ\iso\Hom (\BF_p^\times ,E^\times ), \quad j\mapsto\tau^j\, .
\]

\subsubsection{The local systems $\E_{\mathbf c}$}  \label{sss:Katz sheaf}
Let $\rho_1,\ldots,\rho_n$ and $\chi_1,\ldots,\chi_n$ be two unordered lists of characters $\BF_p^\times\to E^\times$ such that $\rho_i\ne\chi_j$ for all $i$ and $j$. Given these data, N.Katz defines in \cite[Ch.~8]{K90} a certain lisse $E_\lambda$-sheaf on $\BG_m\setminus\{ 1\}$ of rank $n$, which he calls hypergeometric; he gives a brief summary  in \cite[\S 4]{Ka2}.

We will need only the case that $\chi_1=\ldots =\chi_n=1$. In this case the input data is an unordered list $\mathbf c$ of elements $c_1,\ldots ,c_n\in\{1,\ldots ,p-2\}$. We set $\rho_i:=\tau^{c_i}$, where $\tau$ is the Teichm\"uller character; note that $\rho_i\ne 1$ for all $i$. Katz defines the corresponding local system $\E_{\mathbf c}$ on $\BG_m\setminus\{ 1\}$ as follows. Let $\cL_i$ denote the lisse $E_\lambda$-sheaf on $\BG_m$ corresponding to $\rho_i$ in the usual sense (the ``trace of geometric Frobenius" function on $\BF_{p^m}^\times$ corresponding to $\cL_i$ equals the composition of $\rho_i$ with the norm map $\BF_{p^m}^\times\to\BF_p^\times$). Let $\cF_i$ denote the sheaf on $\BG_m$ defined by $\cF_i:=j_!f^*\cL_i$, where $f:\BG_m\setminus\{ 1\}\mono\BG_m$ is the map $x\mapsto 1-x$ and $j:\BG_m\setminus\{ 1\}\mono\BG_m$ is the natural embedding. Finally, consider the object of the derived category of sheaves on $\BG_m\setminus\{ 1\}$ defined by
\begin{equation}  \label{e:E_c}
\E_{\mathbf c}:=j^*(\cF_1\star_!\ldots\star_!\cF_n)[n-1], 
\end{equation}
where $\star_!$ denotes multiplicative convolution with compact support. 
Katz proves the following statements.

\begin{theorem}   \label{t:Katz}
(i) $\E_{\mathbf c}$ is a lisse $E_\lambda$-sheaf on $\BG_m\setminus\{ 1\}$ of rank $n$. It is irreducible; moreover, 
$\E_{\mathbf c}\otimes_{E_\lambda}\bar E_\lambda$ is geometrically irreducible.

(ii) The canonical morphism $\cF_1\star_!\ldots\star_!\cF_n\to\cF_1\star\ldots\star\cF_n$ is an isomorphism (here $\star$ denotes multiplicative convolution without compact support).

(iii) $\E_{\mathbf c}$ is tamely ramified at $0,1,\infty\in\BP^1$.

(iv)  $\E_{\mathbf c}$ is pure of weight $n-1$.

(v) Set $N=n(n-1)/2$; then the rank $1$ local system $(\det\E_{\mathbf c})(N)$ has finite order.
\end{theorem}

The theorem is proved in \cite[Ch.~8]{K90}: statements (i)-(iv) correspond to parts (1), (4), (5), and (8) of  \cite[Thm.~8.4.2]{K90}, and statement (v) corresponds to parts (1a)-(1b) of  \cite[Thm.~8.12.2]{K90}. Some technical issues are explained in \S\ref{sss:Katznotation} below.

\begin{rem}   
Katz also describes the local monodromies of $\E_{\mathbf c}$ at $0,1,\infty\in\BP^1$. In particular, he proves that the local monodromy at 
$0\in\BP^1$ is maximally unipotent and the eigenvalues of the local monodromy at $\infty$ correspond to the characters $\rho_1,\ldots,\rho_n$.
\end{rem}

\begin{rem}   \label{r:duality}
Theorem~\ref{t:Katz}(ii) implies that the local system dual to $\E_{\mathbf c}$ is isomorphic to $\E_{\mathbf c'}(n-1)$, where $\mathbf c'$ is the $n$-uple formed by the numbers $c'_i:=p-1-c_i$.
\end{rem}

The next subsection should be skipped by the reader unless he wants to check that Theorem~\ref{t:Katz} is indeed proved in \cite[Ch.~8]{K90}.

\subsubsection{Some comments on  \cite{K90} and  \cite{Ka2}}   \label{sss:Katznotation}
In  \cite[\S 8.2.2]{K90} Katz introduces a complex of sheaves on $\BG_m$ denoted by $\Hyp (!, \psi;\chi\mbox{'s};\rho\mbox{'s })$ (here $\psi$ is a non-trivial additive character of $\BF_p$). Its multiplicative translate by 
$\lambda\in\BG_m$ is denoted in \cite[\S 8.2.13]{K90} by $\Hyp_\lambda(!, \psi; \chi\mbox{'s}; \rho\mbox{'s})$.
In \cite[Thm.~8.4.2]{K90} he sets $\cH_\lambda (!, \psi; \chi\mbox{'s}; \rho\mbox{'s}):=\Hyp_\lambda (!, \psi; \chi\mbox{'s}; \rho\mbox{'s})[-1]$ and shows that $\cH_\lambda$ is a sheaf.

In \cite[\S 4]{Ka2} Katz considers only $\lambda=1$ and writes $\cH$ instead of $\cH_1$. On p.98 of \cite{Ka2} he introduces a sheaf 
$\cH^{can} (\chi\mbox{'s}; \rho\mbox{'s})$, which is a tensor product of $\cH$ and a rank $1$ local system on $\Spec\BF_p$; unlike $\cH$,
the sheaf $\cH^{can}$ does not depend on $\psi$. On p.99 of \cite{Ka2} he writes a formula\footnote{By \cite[\S 8.2.3]{K90}, it suffices to prove this formula for $n=1$. This can be done by comparing the  ``trace of Frobenius" functions. } for $\cH^{can}$ which does not involve $\psi$ at all. This formula shows that if the $\chi_i$'s are trivial then the restriction of $\cH^{can} (\chi\mbox{'s}; \rho\mbox{'s})$ to $\BG_m\setminus\{ 1\}$ is the sheaf that we denote by $\E_{\mathbf c}$.

To deduce Theorem~\ref{t:Katz}(iv-v)  from the results of \cite{K90}, one has to take into account that our $\E_{\mathbf c}$ corresponds to $\cH^{can}$, while Theorems~8.4.2 and 8.12.2 of \cite{K90} are formulated in terms of $\cH$; one also has to keep in mind that in our situation the $\chi_i$'s are trivial.

\subsubsection{The crystalline companions of $\E_{\mathbf c}$}   \label{sss:Miyatani}
According to K.~Miyatani \cite{Mi}, for any $\mathbf c$ as in \S\ref{sss:Katz sheaf}  there exists an irreducible object $M_{\mathbf c}\in\FIsocd (\BG_m\setminus \{ 1\})$ such that for every closed point $x\in\BG_m\setminus \{ 1\}$ the Frobenius characteristic polynomial of $(M_{\mathbf c})_x$ is equal to that of $(\E_{\mathbf c})_x$ (the word ``equal" makes sense because the coefficients of the Frobenius characteristic polynomial of $(\E_{\mathbf c})_x$ are in the number field $E$, which is a subfield of $\BQ_p$). The construction of $M_{\mathbf c}$ given in \cite[\S 3.2]{Mi} is parallel to the one from \S\ref{sss:Katz sheaf}.

\subsection{The counterexamples}
For any $x\in\bar\BF_p^\times\setminus \{ 1\}$ the characteristic polynomial of the geometric Frobenius (with respect to the field $\BF_p (x)\subset\bar\BF_p$) acting on the stalk $(\E_{\mathbf c})_x$ has coefficients in $E$.
Since $E\subset\BQ_p$ we can talk about the slopes of $\E_{\mathbf c}$ at $x$. We denote them by $a_i^x(\E_{\mathbf c} )$, where $a_1^x(\E_{\mathbf c})\ge\ldots\ge a_n^x(\E_{\mathbf c})$.

\begin{prop}  \label{p:counterexample}
Assume that $p\ge 5$. Let $n=3$ and $\mathbf c=(c_1,c_2,c_3)$, where $c_1=1$, $c_2=p-2$, and $c_3$ is any element of $\{ 1,\ldots ,p-2\}$ different from $(p-1)/2$. Let $\E_{\mathbf c}$ be the lisse $E_\lambda$-sheaf  of rank~3 on $\BG_m\setminus\{ 1\}$  defined in \S\ref{sss:Katz sheaf}. Then there is a unique $x\in\BF_p^\times\setminus\{ 1\}$ such that $a_1^x(\E_{\mathbf c})-a_2^x(\E_{\mathbf c})>1$;
namely, $x$ is the residue class of $-(2c_3)^{-1}$. 
\end{prop}

The proof will be given in \S\ref{ss:counterexample proof}.

The local systems $\E_{\mathbf c}$ from Proposition~\ref{p:counterexample} are the
counterexamples promised in Remark~\ref{r:not guaranteed}, and their crystalline companions (see \S\ref{sss:Miyatani}) are the counterexamples promised in \S\ref{ss:counterex}.

\subsection{The key computation} 
Just as in \S\ref{sss:Katz sheaf}, let $n$ be any positive integer and $\mathbf c:=(c_1,\ldots ,c_n)$, where $c_i\in \{ 1,\ldots ,p-2\}$. 
Given $m\in\BN$ and $x\in\BF_{p^m}^\times\setminus \{ 1\}$ set 
$$P_x^{(m)}(t):=\det (1-F_xt,(\E_{\mathbf c})_x),$$ 
where $(\E_{\mathbf c})_x$ is the stalk of $\E_{\mathbf c}$ at $x\in\bar\BF_p^\times\setminus \{ 1\}$ and $F_x$ is the geometric Frobenius of the field $\BF_{p^m}$. By~\eqref{e:E_c}, the coefficients of the polynomial $P_x^{(m)}$ belong to the ring of integers $O_E$, which is a subring of $\BZ_p$. Let 
$\bar P_x^{(m)}$ be the image of $P_x^{(m)}$ in $\BZ_p [t]/p\BZ_p [t]=\BF_p [t]$. 

\begin{prop}  \label{p:key formula}
(i) One has
\begin{equation}  \label{e:log derivative}
-\frac{d}{dt}\log \bar P_x^{(m)}(t)=\alpha_x^{(m)} (1-\alpha_x^{(m)}t)^{-1}, \quad \alpha_x^{(m)}:=N_{\BF_{p^m}/\BF_p} (u_{\mathbf c}(x)),
\end{equation}
where $N_{\BF_{p^m}/\BF_p}:\BF_{p^m}\to\BF_p$ is the norm map and $u_{\mathbf c}\in\BF_p[X] $ is the following polynomial:
\begin{equation}   \label{e:thepolynomial}
u_{\mathbf c}(X):=\sum_{r\ge 0}(-1)^{nr}\binom{c_1}{r}\cdot\ldots\cdot\binom{c_n}{r}X^r .
\end{equation}

(ii) If $p>n$ then $\bar P_x^{(m)}(t)=1-\alpha_x^{(m)}t$.
\end{prop}

Later we will show that the assumption $p>n$ in statement (ii) is unnecessary, see Corollary~\ref{c:slopes for any n}.

The proof given below is somewhat similar to that of the Chevalley-Warning theorem.
 
\begin{proof}
$\bar P_x^{(m)}$ is a polynomial in $t$ of degree $\le n$, so it can be reconstructed from its logarithmic derivative if $p>n$. Therefore it suffices to prove (i). 

For $m\in\BN$ and $x\in \BF_p^\times$ set 
\[
V_x^{(m)}:=\{(x_1,\ldots ,x_n)\in (\BF_{p^m}^\times)^n\,|\,x_1\cdot\ldots\cdot x_n=x\},
\]
\[
S^{(m)}(x):=\sum_{\,\,V_x^{(m)} }(1-x_1)^{\tilde c_1}\cdot\ldots\cdot (1-x_n)^{\tilde c_n}, 
\]
where $\tilde c_i:=c_i (1+p+\ldots +p^{m-1})$.
By formula~\eqref{e:E_c}, statement (i) is equivalent to the following one: for any $m\in\BN$ and $x\in \BF_{p^m}^\times$ one has 
$S^{(m)}(x)=(-1)^{n-1} \alpha_x^{(m)}$.

Note that  $\tilde c_i<p^m-1$ for all $i$. Also note that if $r_1,\ldots r_n\in\{0,1,\ldots ,p^m-2\}$ then the sum 
\[
\sum\limits_{V_x^{(m)}}\prod\limits_{i=1}^n x_i^{r_i}
\]
is non-zero only if $r_1=\ldots=r_n$; in this  case it equals $(-1)^{n-1}x^r$, where $r:=r_1=\ldots=r_n$. So 
$S^{(m)}(x)=u_{ \tilde{ \mathbf c}}(x)$, where $u_{ \tilde{ \mathbf c}}$ is defined by formula~\eqref{e:thepolynomial} with $\mathbf c$ replaced by $\tilde{ {\mathbf c}}:=(\tilde{c}_1, \dots, \tilde{c}_n)$.  On the other hand, 
\[
\alpha_x^{(m)}:=N_{\BF_{p^m}/\BF_p} (u_{\mathbf c}(x))=\prod\limits_{j=0}^{m-1}u_{\mathbf c}(x)^{p^j}.
\]
So it remains to prove the following identity in $\BF_p[X]$:
\[
u_{ \tilde{ \mathbf c}}(X)=\prod\limits_{j=0}^{m-1}u_{\mathbf c}(X)^{p^j}. 
\]
To show this, use the following well known property of binomial coefficients: if $N=\sum\limits_{j=0}^{m-1}N_jp^j$ and 
$r=\sum\limits_{j=0}^{m-1}r_jp^j$, where $N_j\, ,r_j\in\{ 0,1,\ldots , p-1\}$, then 
\[
\binom{N}{r}\equiv\prod\limits_{j}\binom{N_j}{r_j}\mod p.
\]
This property follows from the identity $(1+X)^N=\prod\limits_{j}(1+X^{p^j})^{N_j}$ in $\BF_p [X]$.
\end{proof}

\begin{cor}   \label{c:slopes for any n}
Let $n>1$ and $\mathbf c:=(c_1,\ldots ,c_n)$, where $c_i\in \{ 1,\ldots ,p-2\}$. 
The equality 
\begin{equation} \label{e:P_x}
\bar P_x^{(m)}(t)=1-\alpha_x^{(m)}t
\end{equation}
 from Proposition~\ref{p:key formula} holds without assuming that $p>n$. The slopes $a_i^x(\E_{\mathbf c} )$, 
$x\in\bar\BF_p^\times\setminus\{ 1\}$, have the following properties:

(a) $a_i^x(\E_{\mathbf c} )+a_{n+1-i}^x(\E_{\mathbf {c'}} )=n-1$, where $\mathbf c'$ is as in Remark~\ref{r:duality};

(b) $\sum\limits_{i=1}^n a_i^x(\E_{\mathbf c} )=n(n-1)/2$;

(c) $a_n^x(\E_{\mathbf c} )\ge 0$, $a_{n-1}^x(\E_{\mathbf c} )> 0$,  and
\[
a_n^x(\E_{\mathbf c} )> 0 \Leftrightarrow u_{\mathbf c}(x)=0,
\]
where $u_{\mathbf c}$ is the polynomial defined by~\eqref{e:thepolynomial}.

(d) $a_1^x(\E_{\mathbf c} )\le n-1$, $a_2^x(\E_{\mathbf c} )< n-1$, and
\[
a_{1}^x(\E_{\mathbf c} )< n-1 \Leftrightarrow u_{\mathbf c'}(x)=0,
\]
where $\mathbf c'$ is as in Remark~\ref{r:duality}.

(e) For all but finitely many $x\in\bar\BF_p^\times\setminus\{ 1\}$ one has $a_i^x(\E_{\mathbf c} )=n-i$.
\end{cor}

\begin{proof}
Statement (a) follows from Remark~\ref{r:duality}. Statement (b) follows from Theorem~\ref{t:Katz}(v). By~\eqref{e:E_c}, the eigenvalues of the geometric Frobenius acting on the stalks of $\E_{\mathbf c}$ are algebraic integers, so 
\begin{equation} \label{e: ge 0}
 a_n^x(\E_{\mathbf c} )\ge 0 \quad \mbox{ for all }x. 
\end{equation} 

Let us prove (e). The polynomial $u_{\mathbf c}$ is non-zero (in fact, its constant term equals 1). So formula~\eqref{e:log derivative} implies that for all but finitely many $x$ one has $a_n^x(\E_{\mathbf c} )=0$. By (a), this implies that $a_1^x(\E_{\mathbf c} )=n-1$ for almost all $x$. On the other hand, combining Theorem~\ref{t:l-adic}(iii) with Theorem~\ref{t:Katz}(i), we see that $a_i^x(\E_{\mathbf c})-a_{i+1}^x(\E_{\mathbf c})\le 1$ for almost all $x$. Statement (e) follows.

By (b),(e), and Theorem~\ref{t:l-adic}(ii), $a_{n-1}^x(\E_{\mathbf c} )+a_n^x(\E_{\mathbf c} )>0$ for all $x$. Since $a_{n-1}^x(\E_{\mathbf c} )\ge a_n^x(\E_{\mathbf c} )$, we see that $a_{n-1}^x(\E_{\mathbf c} )>0$ for all $x$. This means that all polynomials $\bar P_x^{(m)}$ have degree $\le 1$. Combining this with~\eqref{e:log derivative}, we get formula~\eqref{e:P_x}.

Statement (c) follows from \eqref{e: ge 0} and \eqref{e:P_x}. Statement (d) follows from (a) and~(c). 
\end{proof}

\subsection{Proof of Proposition~\ref{p:counterexample}}  \label{ss:counterexample proof}
We apply Corollary~\ref{c:slopes for any n}. In the situation of Proposition~\ref{p:counterexample} the polynomial $u_{\mathbf c}$ defined by~\eqref{e:thepolynomial} has degree $1$. Its unique root $x_0$ equals $-(2c_3)^{-1}$, which is not equal to~$1$ and not a root of $u_{\mathbf c'}$. So by Corollary~\ref{c:slopes for any n}(c,d) we have 
\[
a_3^{x_0}(\E_{\mathbf c} )>0, \quad a_3^{x}(\E_{\mathbf c} )=0 \mbox{ for } x\ne x_0\, ,
\]
\[
a_1^{x_0}(\E_{\mathbf c} )=2, \quad a_1^{x}(\E_{\mathbf c} )\le 2  \mbox{ for all } x\, .
\]
By Corollary~\ref{c:slopes for any n}(b), $a_1^{x}(\E_{\mathbf c} )+a_2^{x}(\E_{\mathbf c} )+a_3^{x}(\E_{\mathbf c} )=3$. Thus $a_1^{x_0}(\E_{\mathbf c} )=2$ and 
$a_2^{x_0}(\E_{\mathbf c} )<1$; on the other hand, if $x\ne x_0$ then $a_1^x(\E_{\mathbf c} )+a_2^x(\E_{\mathbf c} )=3$ and $a_1^{x}(\E_{\mathbf c} )\le 2$. 
So $a_1^x(\E_{\mathbf c} )-a_2^x(\E_{\mathbf c} )>1$ if and only if $x=x_0$.
 \qed

\subsection{Using a computer to produce more examples} \label{ss:Magma}
In the situation of Proposition~\ref{p:counterexample} the field generated by the coefficients of the Frobenius characteristic polynomials is not equal to $\BQ$. One can produce examples of sheaves $\E_{\mathbf c}$ such that this field equals $\BQ$ and 
$a_i^x(\E_{\mathbf c})-a_{i+1}^x(\E_{\mathbf c})>1$ for some $i$ and~$x$. However, this requires taking $n \geq 4$, which seems to necessitate the use of computer calculations.

For example\footnote{The sheaf $\E_{\mathbf c}$ from this example is related to the Dwork pencil of quintic threefolds, see \cite[Thm.~5.3]{Ka2}.
}, take $n=4$ and
\[
c_i=i\cdot (p-1)/5, \quad 1\le i\le 4,
\]
where $p \equiv 1 \mod 5$ (so $\rho_1,\dots,\rho_4$ are the primitive characters of order 5). Note that by Remark~\ref{r:duality} one has
\begin{equation}  \label{e:selfdual}
\E_{\mathbf c}^*\simeq\E_{\mathbf c}(3).
\end{equation}
For $x\in\BF_p^\times\setminus\{ 1\}$ write
\[
\det (1-F_xt,(\E_{\mathbf c})_x)=\sum_{i=0}^4 b_i^x t^i, \quad b_i^x\in\BZ, \, b_0^x=1.
\]
Using Proposition~\ref{p:key formula}(ii) and a simple computer search, one easily identifies pairs $(p,x)$ for which $p | b_1^x$; 
this also ensures that $p | b_2^x$ by the same Proposition~\ref{p:key formula}(ii). Using \eqref{e:selfdual}, one sees that for such pairs $(p,x)$ the slopes are $\left( \frac{5}{2}, \frac{5}{2}, \frac{1}{2}, \frac{1}{2}\right)$ provided that $b_2^x$ is not divisible by $p^2$. One cannot check this non-divisibility using Proposition~\ref{p:key formula}(ii), but a related $p$-adic calculation
is implemented in the ``hypergeometric motives" package\footnote{To specify a rank $n$ hypergeometric local system for \textsc{Magma}, one enters two sequences of rational numbers of length $n$ (they correspond to the eigenvalues of the local monodromy at $0$ and $\infty$). In our situation the sequences are $[0,0,0,0]$ and $[1/5,2/5,3/5,4/5]$.} of the \textsc{Magma} computer algebra system (see \cite{Mag}).
For example, \textsc{Magma} confirms that for $p=31$ and $x$ equal to $4$ or $17$, the slopes equal  $\left( \frac{5}{2}, \frac{5}{2}, \frac{1}{2}, \frac{1}{2}\right)$.

\section{Recollections on the Tannakian categories $\FIsoc (X)$ and $\FIsocd (X)$ }  \label{s:Crew}

In this Appendix we recall some results of R.~Crew \cite{Crew-mono}. 

\subsection{Notation}
\subsubsection{}
We fix a universal cover $\tX\to X$ and set $\Pi:=\Aut (\tX /X)$.

\subsubsection{}
The category of finite-dimensional vector spaces over a field $E$ is denoted by $\Vect_E$.
The category of finite-dimensional representations of $G$ is denoted by $\Rep_E(G)$.  

\subsubsection{}   \label{sss:equiv objects}
If a finite group $\Gamma$ acts on a Tannakian category $\cT$ then $\Gamma$-equivariant objects of $\cT$ form a Tannakian category; we denote it by $\cT^\Gamma$. For instance, $(\Vect_E)^\Gamma=\Rep_E (\Gamma )$.

\subsection{Isotrivial $F$-isocrystals}
\subsubsection{Definition of isotriviality}
An object $M\in\FIsoc (X)$ is said to be \emph{trivial} if it is isomorphic to a direct sum of several copies of the unit object. An object 
$M\in\FIsoc (X)$ is said to be \emph{isotrivial} if its pullback to $\tX /U$ is trivial for some open subgroup $U\subset\Pi$. Isotrivial objects form a Tannakian subcategory\footnote{Following \cite[\S 2.3.5]{An},  by a \emph{Tannakian subcategory} of a Tannakian category $\cT$ we mean a strictly full subcategory $\cT'\subset\cT$ stable under tensor products, direct sums, dualization, and passing to subobjects.} of $\FIsoc (X)$.

\subsubsection{}
Let $U\subset\Pi$ be an open normal subgroup. Then $\Pi/U$ acts on $\tX /U$ and therefore on $\FIsoc (\tX /U)$. So we have the Tannakian category $\FIsoc (\tX /U)^{\Pi/U}$ (see \S\ref{sss:equiv objects}).

\begin{lem}   \label{l:isotriv}
Pullback from $X$ to $\tX /U$ defines a tensor equivalence
\begin{equation}  \label{e:etale descent}
\FIsoc (X)\iso\FIsoc (\tX /U)^{\Pi/U}.
\end{equation}
\end{lem}

\begin{proof}
Follows from etale descent for $F$-isocrystals.
\end{proof}

\begin{rem}
Note that $\FIsoc (\tX /U)^{\Pi/U}\supset(\Vect_{\BQ_p})^{\Pi/U}=\Rep_{\BQ_p}(\Pi/U)$. It is clear that the 
equivalence~\eqref{e:etale descent} identifies $\Rep_{\BQ_p}(\Pi/U)$ with the full subcategory of those objects of 
$\FIsoc (X)$ whose pullback to $\tX$ is trivial.
\end{rem}

\begin{rem}   \label{r:isotrivial smooth}
By the previous remark, the Tannakian subcategory of isotrivial objects of $\FIsoc (X)$ canonically identifies with 
$\Rep^{\smooth}_{\BQ_p}(\Pi)$, where $\Rep^{\smooth}_{\BQ_p}(\Pi)$ is the Tannakian category of smooth finite-dimensional representations of $\Pi$ over $\BQ_p$, i.e., 
\begin{equation}   \label{e:smooth as lim}
\Rep^{\smooth}_{\BQ_p}(\Pi):=\underset{U}{\underset{\longrightarrow}\lim} \Rep_{\BQ_p} (\Pi/U).
\end{equation}
So we get a fully faithful embedding $\Rep^{\smooth}_{\BQ_p}(\Pi)\mono \FIsoc (X )$ and therefore a fully faithful embedding 
\begin{equation} \label{e:isotrivial smooth}
\Rep^{\smooth}_{\BQbar_p}(\Pi)\mono \FIsoc (X )\otimes_{\BQ_p}\BQbar_p\, , 
\quad\quad \Rep^{\smooth}_{\BQbar_p}(\Pi):=\Rep^{\smooth}_{\BQ_p}(\Pi)\otimes_{\BQ_p}\BQbar_p ,
\end{equation}
whose essential image is a Tannakian subcategory of $\FIsoc (X )\otimes_{\BQ_p}\BQbar_p$.
\end{rem}

\subsection{Crew's theorem on unit-root $F$-isocrystals}   \label{ss:unit-root}
Recall that an object $M\in\FIsoc (X)$ is said to be \emph{unit-root} (or \emph{etale}) if all its slopes at all points of $X$ (or equivalently, at the generic point $\eta\in X$) are zero. The Tannakian subcategory of unit-root objects of 
$\FIsoc (X)$ is denoted by $\FIsoc_{\et} (X)$. All isotrivial $F$-isocrystals are unit-root.

Let $\Rep^{\cont}_{\BQ_p}(\Pi)$ denote the Tannakian category of \emph{continuous} representations of $\Pi$ in finite-dimensional vector spaces over $\BQ_p$. It contains  $\Rep^{\smooth}_{\BQ_p}(\Pi)$ as a Tannakian subcategory. In Remark~\ref{r:isotrivial smooth} we defined a tensor equivalence between $\Rep^{\smooth}_{\BQ_p}(\Pi)$ and the category of isotrivial $F$-isocrystals. According to \cite[Thm.~2.1]{Cr}, it \emph{extends to a canonical tensor equivalence} 
\[
\Rep^{\cont}_{\BQ_p}(\Pi)\iso\FIsoc_{\et} (X).
\]

\subsection{Crew's characterization of the Tannakian subcategory $\Rep^{\smooth}_{\BQ_p}(\Pi)\subset\FIsoc (X)$}
By \eqref{e:smooth as lim}, $\Rep^{\smooth}_{\BQ_p}(\Pi)$ is a union of an increasing family of Tannakian categories of the form $\Rep_{\BQ_p}(\Gamma)$, where $\Gamma$ is a finite group. On the other hand, one has the following result. 

\begin{prop}   \label{p:Crew}
Let $\Gamma$ be a finite group. Then any tensor functor $\Rep_{\BQ_p}(\Gamma )\to\FIsoc (X)$ factors through
$\Rep^{\smooth}_{\BQ_p}(\Pi)\subset\FIsoc (X)$. Moreover, any tensor functor $\Rep_{\BQbar_p}(\Gamma )\to\FIsoc (X)\otimes_{\BQ_p}\BQbar_p$ factors through $\Rep^{\smooth}_{\BQbar_p}(\Pi)\subset\FIsoc (X)\otimes_{\BQ_p}\BQbar_p$. 
\end{prop}
\begin{proof}
It suffices to prove the second statement. Let $\Phi:\Rep_{\BQbar_p}(\Gamma )\to\FIsoc (X)\otimes_{\BQ_p}\BQbar_p$ be a tensor functor and $\im \Phi$  its essential image. We will first prove that
\begin{equation}   \label{e:contained in unit-root}
\im \Phi\subset \FIsoc_{\et} (X)\otimes_{\BQ_p}\BQbar_p \, .
\end{equation}
Proving this amounts to showing that for any algebraically closed field $L$ and any morphism $\alpha :\Spec L\to X$
one has $\im (\alpha^*\circ \Phi )\subset \FIsoc_{\et} (\Spec L)\otimes_{\BQ_p}\BQbar_p$. Note that 
$\FIsoc (\Spec L)\otimes_{\BQ_p}\BQbar_p \simeq\Rep_{\BQbar_p}(\widetilde\BG_m)$, where $\widetilde\BG_m$ is the pro-torus with group of characters $\BQ$. So $\alpha^*\circ \Phi$ is a tensor functor $\Rep_{\BQbar_p}(\Gamma )\to \Rep_{\BQbar_p}(\widetilde\BG_m )$. It corresponds to a homomorphism $\widetilde\BG_m\to\Gamma$. It remains to note that all such homomorphisms are trivial.

Thus we have proved \eqref{e:contained in unit-root}. So by \S\ref{ss:unit-root}, it remains to show that any tensor functor 
\begin{equation}   \label{e:our tens funct}
\Rep_{\BQbar_p}(\Gamma )\to \Rep^{\cont}_{\BQ_p}(\Pi)\otimes_{\BQ_p}\BQbar_p
\end{equation}
factors through $\Rep^{\smooth}_{\BQbar_p}(\Pi)$. Define an affine group scheme $\hat\Pi$ over $\BQ_p$ as follows: 
$\hat\Pi$ is the projective limit of all affine algebraic group schemes $H$ over $\BQ_p$ equipped with a continuous homomorphism $\Pi\to H(\BQ_p )$ with Zariski-dense image. Then $\Rep^{\cont}_{\BQ_p}(\Pi)=\Rep_{\BQ_p}(\hat\Pi)$. So a tensor functor \eqref{e:our tens funct} corresponds to a homomorphism $\hat\Pi\to\Gamma$ or equivalently, to a continuous homomorphism $\Pi\to\Gamma$. The latter has open kernel, which finishes the proof.
\end{proof}

\subsection{The group $\pi_1^{\FIsoc} (X)$}   \label{ss:crys fund group}
Warning: the group $\pi_1^{\FIsoc} (X)$ defined below is \emph{different} from (but closely related to) the group denoted by $\pi_1^{\FIsoc} (X)$ in formula (2.5.4) on p.~446 of \cite{Crew-mono}. 

Set 
$$\FIsoc (\tX ):=\underset{U}{\underset{\longrightarrow}\lim} \FIsoc (\tX /U),$$
where $U$ runs through the set of open subgroups of $\Pi$. Fix a fiber functor 
$\tilde\xi :\FIsoc (\tX )\otimes_{\BQ_p}\BQbar_p\to\Vect_{\BQbar_p}$. The existence of $\tilde\xi$ is guaranteed by a general theorem of Deligne~\cite{De3}; one can also construct $\tilde\xi$ by choosing a closed point $\tilde x\in\tX$ and a fiber functor on $\FIsoc (\tilde x )\otimes_{\BQ_p}\BQbar_p$.

Let $\xi :\FIsoc (X )\otimes_{\BQ_p}\BQbar_p\to\Vect_{\BQbar_p}$ be the composition of $\tilde\xi$ with the pullback functor from $X$ to $\tX$. We set $\pi_1^{\FIsoc} (X):=\Aut \xi$; this is an affine group scheme over $\BQbar_p$, and one has a canonical equivalence 
$\FIsoc (X )\otimes_{\BQ_p}\BQbar_p \iso\Rep_{\BQbar_p} (\pi_1^{\FIsoc} (X))$.
The fully faithful functor \eqref{e:isotrivial smooth}
defines a canonical epimorphism 
\begin{equation}   \label{e:crys to usual}
\pi_1^{\FIsoc} (X)\epi\Pi. 
\end{equation}

For any open subgroup $U\subset\Pi$ one has a similar group $\pi_1^{\FIsoc} (\tX /U)$ equipped with an epimorphism
$\pi_1^{\FIsoc} (\tX /U)\epi U$. If $U\subset U'$ then pullback from $\tX/U'$ to $\tX/U$ defines a homomorphism
$\pi_1^{\FIsoc} (\tX /U)\to\pi_1^{\FIsoc} (\tX /U')$ over $\Pi$. In particular, one has a canonical homomorphism
\begin{equation}  \label{e:pi_1 to pi_1}
\pi_1^{\FIsoc} (\tX /U)\to\pi_1^{\FIsoc} (X).
\end{equation}
\begin{prop}    \label{p:pi_1 of cover}
The homomorphism \eqref{e:pi_1 to pi_1} identifies $\pi_1^{\FIsoc} (\tX /U)$ with $\pi_1^{\FIsoc} (X)\times_{\Pi}U$.
\end{prop}

\begin{proof}
Choose an open normal subgroup $V\subset\Pi$ so that $V\subset U$.
Lemma~\ref{l:isotriv} implies that the sequences
\[
0\to \pi_1^{\FIsoc} (\tX /V)\to\pi_1^{\FIsoc} (X)\to\Pi/V\to 0,
\]
\[
0\to \pi_1^{\FIsoc} (\tX /V)\to\pi_1^{\FIsoc} (\tX /U)\to U/V\to 0
\]
are exact. The proposition follows. 
\end{proof}

\begin{prop}  \label{p:kernel connected}
The kernel of the canonical epimorphism $\pi_1^{\FIsoc} (X)\epi\Pi$ is connected. In other words, 
$\Ker (\pi_1^{\FIsoc} (X)\epi\Pi)=((\pi_1^{\FIsoc} (X))^\circ$.
\end{prop}

\begin{proof}
By Proposition~\ref{p:Crew}, every morphism from $\pi_1^{\FIsoc} (X)$ to a finite group factors through~$\Pi$.
\end{proof}

\subsubsection{Some homomorphisms $\widetilde\BG_m\to\pi_1^{\FIsoc} (X)$}  \label{sss:Some homs}
Let $x\in X$, and let $x_{\perf}$ denote the spectrum of the perfection of the residue field of $x$. The Tannakian category $\FIsoc (x_{\perf})$ has a canonical $\BQ$-grading by slopes, so we have a canonical central homomorphism $\widetilde\BG_m\to\pi_1^{\FIsoc}(x_{\perf})$. Composing it with the homomorphism
\begin{equation}   \label{e:pi1 to pi1}
\pi_1^{\FIsoc}(x_{\perf})\to\pi_1^{\FIsoc}(X),
\end{equation}
we get a homomorphism
\begin{equation}    \label{e:tGm to pi1}
\widetilde\BG_m\to\pi_1^{\FIsoc}(X);
\end{equation}
note that the homomorphisms \eqref{e:pi1 to pi1}-\eqref{e:tGm to pi1} are defined only up to $\pi_1^{\FIsoc}(X)$-conjugacy.

Let $U\subset\Pi$ be an open subgroup and suppose that we fix $x'\in\tX/U$ such that $x'\mapsto x$. Applying the previous procedure to 
$(\tX/U ,x')$ instead of $(X,x)$, one gets a homomorphism \eqref{e:tGm to pi1} defined up to $\pi_1^{\FIsoc}(\tX/U)$-conjugacy.

Passing to the limit with respect to $U$, we get for each $\tilde x\in\tX$ a homomorphism
\begin{equation}   \label{e:refined}
\nu_{\tilde x}:\widetilde\BG_m\to\pi_1^{\FIsoc}(X)
\end{equation}
defined up to conjugation by elements of the group $\Ker (\pi_1^{\FIsoc} (X)\epi\Pi)$. This group is the neutral connected component
$(\pi_1^{\FIsoc} (X))^\circ$, see Proposition~\ref{p:kernel connected}.

\subsection{A general lemma on $\Rep_E(G)$}
\begin{lem}   \label{l:Tannakian nonsense}
Let $f:H\to G$ be a homomorphism of affine group schemes over a field $E$ and $f^*:\Rep_E (G)\to\Rep_E  (H)$ the corresponding tensor functor.

(i) $f^*$ is fully faithful if and only if every regular function on $G/f(H)$ is constant. 

(ii) $f$ is an epimorphism\footnote{By definition, ``epimorphism" means that $f(H)$ equals $G$ as a scheme. If $E$ has characteristic $0$ this just means that $f$ is surjective.} if and only if $f^*$ has the following property: for every $V\in\Rep_E (G)$ every $H$-submodule of $V$ is a 
$G$-submodule.
\end{lem}

Probably the lemma is  well known: e.g., statement (ii) is [DM,Prop.2.21 (a)]. We give a proof for completeness.

\begin{proof}
(i) Full faithfulness means that for every $V,W\in\Rep (G)$ the map
\begin{equation}   \label{e:res to H}
\Hom_G (V,W)\to\Hom_H(V,W)
\end{equation}
is an isomorphism. Let $L$ denote the space of regular functions on $G/f(H)$, then the map \eqref{e:res to H} is just the map
\[
\Hom_G(V\otimes W^*,E)\to\Hom_G(V\otimes W^*,L)
\]
induced by the inclusion $E\mono L$. So the map \eqref{e:res to H} is an isomorphism if and only if $L=E$.

(ii) It suffices to prove the ``if" statement. Suppose that for every $V\in\Rep_E (G)$ every $H$-submodule $W\subset V$ is a 
$G$-submodule. Then this is true even if $V$ is an infinite-dimensional $G$-module. Take $V$ to be the space of regular functions on $G$ (on which $G$ acts by left translations), and let $W\subset V$ be the ideal of $f(H)$. We see that $W$ is a $G$-submodule, which means that $f(H)=G$.
\end{proof}

\begin{ex}   \label{ex:parabolic}
Let $G$ be a connected reductive group over a field $E$ and $P\subset G$ a parabolic subgroup. By Lemma~\ref{l:Tannakian nonsense}(i), the restriction functor $\Rep_E (G)\to\Rep_E  (P)$ is fully faithful, so one can think of $\Rep_E (G)$ as a full subcategory of $\Rep_E  (P)$. However, if $P\ne G$ this subcategory is \emph{not closed under passing to subobjects}: this follows from Lemma~\ref{l:Tannakian nonsense}(ii) or by direct inspection. So if $P\ne G$ then $\Rep_E (G)$ is \emph{not a Tannakian subcategory} of  $\Rep_E (P)$ in the sense of \cite[\S 2.3.5]{An} (although it is a full subcategory which is a Tannakian category). 
\end{ex}

\subsection{Overconvergent $F$-isocrystals and the group $\pi_1^{\FIsocd} (X)$}
Overconvergent $F$-isocrystals form a Tannakian category $\FIsocd (X)$. It is equipped with a tensor functor $\FIsocd (X)\to\FIsoc (X)$. As already mentioned in \S\ref{sss:conv vs overconv}, this functor is known to be fully faithful, so we view $\FIsocd (X)$ as a full subcategory of $\FIsoc (X)$.

\subsubsection{Warning} \label{sss:not Tannakian subcategory}
It often happens that the full subcategory $\FIsocd (X)\subset\FIsoc (X)$ is not closed with respect to passing to subobjects (see \cite[Rem.~5.12]{Ke6}). In this case $\FIsocd (X)$ is \emph{not a Tannakian subcategory} of  $\FIsoc (X)$ in the sense of 
\cite[\S 2.3.5]{An} (this is similar to Example~\ref{ex:parabolic}).

\subsubsection{The embedding $\Rep^{\smooth}_{\BQ_p}(\Pi)\mono \FIsocd (X )$}
The essential image of the fully faithful functor $\Rep^{\smooth}_{\BQ_p}(\Pi)\mono \FIsoc (X )$ from Remark~\ref{r:isotrivial smooth} is contained in $\FIsocd (X )$ (by etale descent for overconvergent $F$-isocrystals, see \cite[Thm.~1]{Et}).
So the essential image of the functor \eqref{e:isotrivial smooth} is contained in $\FIsocd (X )\otimes_{\BQ_p}\BQbar_p$.

Proposition~\ref{p:Crew} remains valid if one replaces $\FIsoc$ by $\FIsocd$ (this follows from Proposition~\ref{p:Crew} itself).

\subsubsection{The homomorphism $\pi_1^{\FIsoc} (X)\to \pi_1^{\FIsocd} (X)$}

Let 
\[
\tilde\xi :\FIsoc (\tX )\otimes_{\BQ_p}\BQbar_p\to\Vect_{\BQbar_p} \quad \mbox { and }\quad \xi :\FIsoc (X )\otimes_{\BQ_p}\BQbar_p\to\Vect_{\BQbar_p}\]
be as in \S\ref{ss:crys fund group}. Recall that $\pi_1^{\FIsoc} (X):=\Aut \xi$.
Let $\xi^\dagger :\FIsocd (X )\otimes_{\BQ_p}\BQbar_p\to\Vect_{\BQbar_p}$ be the restriction of~$\xi$, and set 
\[
\pi_1^{\FIsocd} (X):=\Aut \xi^\dagger.
\]
Both $\pi_1^{\FIsoc} (X)$ and $\pi_1^{\FIsocd} (X)$ are affine group schemes over $\BQbar_p$. 
Restriction from $\FIsoc (X )\otimes_{\BQ_p}\BQbar_p$ to 
 $\FIsocd (X )\otimes_{\BQ_p}\BQbar_p$ defines a canonical homomorphism 
 \begin{equation}   \label{e:FIsoc to FIsocd}
 \pi_1^{\FIsoc} (X)\to \pi_1^{\FIsocd} (X). 
 \end{equation}
 One has a canonical equivalence 
\[
\FIsocd (X )\otimes_{\BQ_p}\BQbar_p \iso\Rep_{\BQbar_p} (\pi_1^{\FIsocd} (X)).
\]

By \S\ref{sss:not Tannakian subcategory}, the homomorphism \eqref{e:FIsoc to FIsocd} is not always surjective. But it has the following weaker property.

\begin{lem}   \label{l:quasisurjectivity}
Let $H$ denote the image of the homomorphism \eqref{e:FIsoc to FIsocd}. Then every regular function on $\pi_1^{\FIsocd} (X)/H$ is constant.
\end{lem}

\begin{proof}
Follows from Lemma~\ref{l:Tannakian nonsense}(i).
\end{proof}

\subsubsection{The epimorphism $\pi_1^{\FIsocd} (X)\epi\Pi$}
The fully faithful functor $$\Rep^{\smooth}_{\BQbar_p}(\Pi)\mono \FIsocd (X )\otimes_{\BQ_p}\BQbar_p$$ induces an epimorphism 
$\pi_1^{\FIsocd} (X)\epi\Pi$, whose composition with \eqref{e:FIsoc to FIsocd} is equal to the epimorphism~\eqref{e:crys to usual}.

\begin{prop}
(i) The kernel of the canonical epimorphism $\pi_1^{\FIsocd} (X)\epi\Pi$ is connected.

(ii) For any open subgroup $U\subset\Pi$, the canonical homomorphism $\pi_1^{\FIsocd} (\tX /U)\to\pi_1^{\FIsocd} (X)$ induces an isomorphism 
$\pi_1^{\FIsocd} (\tX /U)\iso\pi_1^{\FIsocd} (X)\times_{\Pi}U$.
\end{prop}

\begin{proof}
The proposition can be proved similarly to Propositions~\ref{p:pi_1 of cover}-\ref{p:kernel connected}. Statement (i) also follows from 
Proposition~\ref{p:kernel connected} because by Lemma~\ref{l:quasisurjectivity}, the homomorphism 
$\pi_0(\pi_1^{\FIsoc} (X))\to\pi_0( \pi_1^{\FIsocd} (X))$ is surjective. 
\end{proof}

\bibliographystyle{ams-alpha}

\end{document}